\documentclass[a4paper,11 pt,reqno]{amsart} 
\usepackage{a4,amsbsy,amscd,amssymb,amstext,amsmath,mathtools,mathdots,latexsym,oldgerm,enumerate,verbatim,graphicx,ytableau,xy}
\usepackage{url}
\usepackage[hang,flushmargin]{footmisc} 

\makeatletter                                       
\def\l@subsection{\@tocline{2}{0pt}{2.5pc}{2.5pc}{}}
\makeatother                                        

\makeatletter                                                  
\def\chapter{\clearpage\thispagestyle{plain}\global\@topnum\z@ 
\@afterindenttrue \secdef\@chapter\@schapter}
\makeatother

\newtheorem{thmgl} {Theorem}    
\newtheorem{propgl}{Proposition}
\newtheorem{lemgl} {Lemma}
\newtheorem{corgl} {Corollary}
\newtheorem{defgl} {Definition}

\theoremstyle{definition}

\newtheorem{remgl} {Remark}
\newtheorem{remsgl} [remgl]{Remarks}

\newcommand{\mf}{\mathfrak}
\newcommand{\mc}{\mathcal}
\newcommand{\mb}{\mathbb}

\newcommand{\nts}{\negthinspace}     
\newcommand{\Nts}{\nts\nts}

\newcommand{\ov}{\overline}

\newcommand{\ot}{\otimes}           
\newcommand{\la}{\langle}
\newcommand{\ra}{\rangle}

\newcommand{\Hom}{{\rm Hom}}        

\newcommand{\Ext}{{\rm Ext}}

\newcommand{\ind}{{\rm ind}}

\newcommand{\Sym}{{\rm Sym}} 

\let\ttie\t
\newcommand{\tie}[1]{{\let\t\ttie \ttie#1}}
\renewcommand{\t}{\mf{t}}  

\newcommand{\GL}{{\rm GL}}

\newcommand{\SL}{{\rm SL}}

\newcommand{\Sp}{{\rm Sp}} 

\newcommand{\ve}{\varepsilon}

\makeatletter
\def\vcdots{\vbox{\baselineskip4\p@ \lineskiplimit\z@
\kern3\p@\hbox{.}\hbox{.}\hbox{.}\Nts\nts\kern3\p@}}
\makeatother

\xyoption{all}

\begin{document}

\title{Injective and tilting resolutions and a Kazhdan-Lusztig theory for the general linear and symplectic group}

\begin{abstract}
Let $k$ be an algebraically closed field of characteristic $p>0$ and let $G$ be a symplectic or general linear group over $k$. We consider induced modules for $G$ under the assumption that $p$ is bigger than the greatest hook length in the partitions involved. We give explicit constructions of left resolutions of induced modules by tilting modules. Furthermore, we give injective resolutions for induced modules in certain truncated categories. We show that the multiplicities of the indecomposable tilting and injective modules in these resolutions are the coefficients of certain Kazhdan-Lusztig polynomials. We also show that our truncated categories have a Kazhdan-Lusztig theory in the sense of Cline, Parshall and Scott. This builds further on work of Cox-De Visscher and Brundan-Stroppel.
\end{abstract}

\author[R.\ Tange]{Rudolf Tange}
\address
{School of Mathematics,
University of Leeds,
LS2 9JT, Leeds, UK}
\email{R.H.Tange@leeds.ac.uk}

\maketitle
\markright{\MakeUppercase{{Injective and tilting resolutions for induced modules}}}

\section*{Introduction}
In this paper we study modules for the general linear group $\GL_n$ and for the symplectic group $\Sp_n$. This is a continuation of \cite{T} ($\GL_n$) and \cite{LT} ($\Sp_n$) where we described good filtration multiplicities in indecomposable tilting modules and decomposition numbers in terms of certain cap or cap-curl diagrams and codiagrams.
In the present paper we want to use the same combinatorics to define certain Kazhdan-Lusztig polynomials which we then use to give explicit constructions of left resolutions of induced modules by tilting modules and of injective resolutions for induced modules in certain truncated categories. This is based on work of Cox-De Visscher \cite{CdV} and Brundan-Stroppel \cite{BS}. Throughout we assume that $p$ is bigger than the greatest hook length in the partitions involved.

The paper is organised as follows. In Section~\ref{ss.reductive_groups} we describe the necessary background from the theory of reductive groups and their representations and recall some notation and an important result from \cite{T} and \cite{LT} about a quasihereditary algebra for the partial order $\preceq$.
In Section~\ref{ss.arrow_diagrams} we recall the notions of arrow diagrams and cap and cap-curl diagrams from \cite{T} and \cite{LT}. Then we discuss some combinatorial tools to characterise the partial order $\preceq$ in terms of arrow diagrams, and finally we discuss cap and cap-curl diagrams associated to two weights and the codiagram versions. In Section~\ref{ss.translation} we recall the definition of certain translation functors and, in the case of $\Sp_n$, refined translation functors, and we state two important results from \cite{T} and \cite{LT}: Proposition~\ref{prop.trans_equivalence} on ``translation equivalence" and Proposition~\ref{prop.trans_projection} on ``translation projection".

In Section~\ref{s.polynomials} we introduce the Kazhdan-Lusztig polynomials $d_{\lambda\mu}$ and $p_{\lambda\mu}$ which are characteristic $p$ analogues of the equally named polynomials in \cite{BS} and \cite{CdV}. We state and prove certain elementary properties of these polynomials. We do the same for the polynomials $e_{\lambda\mu}$ and $r_{\lambda\mu}$. The main properties are the recursive relations in Lemma~\ref{lem.d_and_e_eqn} and Proposition~\ref{prop.p_and_r_eqn}.
Our polynomials are not compatible with the parabolic setup $(W_p,W)$ from Lusztig's conjecture, but with $(W_p(A_{s_1+s_2-1}),W(A_{s_1-1})\times W(A_{s_2-1}))$ for $\GL_n$ and $(W_p(D_s),W(A_{s-1}))$ for $\Sp_n$. Our combinatorics of arrow diagrams only describes a finite set of coset representatives corresponding to our set $\Lambda_p$.\footnote{Note that Lusztig's conjecture also concerns a finite set of coset representatives corresponding to the Jantzen region.}
In Remark~\ref{rems.polynomials}.3 we show that our polynomials are products of two, combinatorially defined, ``ordinary" parabolic Kazhdan-Lusztig polynomials as in \cite{Boe}, \cite{CdV} and, in the $\GL_n$-case, \cite{BS}. The parabolic type of the latter is $(W(A_{s_1+s_2-1}),W(A_{s_1-1})\times W(A_{s_2-1}))$ for $\GL_n$ and $(W(D_s),W(A_{s-1}))$ for $\Sp_n$. Note that Boe \cite{Boe} has shown that the KL-polynomials of type $(W(D_s),W(A_{s-1}))$ coincide with the KL-polynomials of type $(W(C_s),W(A_{s-1}))$. To write our polynomials as a product of ``ordinary" KL-polynomials we need to split the arrow diagram into two parts: the \emph{associated pair of $(\land,\vee)$-sequences}, see Section~\ref{ss.arrow_diagrams}. This ``splitting" depends on $p$.

In Section~\ref{s.resolutions} we show that the coefficients of the polynomials $p_{\lambda\mu}$ are the multiplicities of indecomposable tilting modules in certain finite left tilting resolutions of induced modules, see Theorem~\ref{thm.tilting_resolution}. Furthermore, we show that the coefficients of the polynomials $r_{\lambda\mu}$ are the multiplicities of indecomposable injective modules in certain finite injective resolutions of induced modules in certain truncated categories, see Theorem~\ref{thm.injective_resolution}. These are analogues of results in \cite{BS} and \cite{CdV}. We also show that the categories $\mc C_\Lambda$ that we consider have a Kazhdan-Lusztig theory in the sense of Cline, Parshall and Scott \cite{CPS}. This means that the dimensions of the $\Ext_G^i(L(\lambda),L(\mu))$ can be espressed in terms of our polynomials $r_{\lambda\mu}$.

In Section~\ref{s.limiting_case} we show that when we fix the residue of $n$ mod $p$, we can derive a stability result for arrow diagrams when $p\gg0$. This shows in particular that our results \cite[Cor to Thm~6.1 and Prop 8.3]{T} and \cite[Cor to Thm~6.1]{LT} on the decomposition numbers of the walled and ordinary Brauer algebra in characteristic $p$ coincide for $p\gg0$ with the description of the decomposition numbers of the walled and ordinary Brauer algebra in characteristic $0$ from Cox and de Visscher \cite[Thm~4.10 and 5.8]{CdV}.

\section{Preliminaries}\label{s.prelim}
\subsection{Reductive groups}\label{ss.reductive_groups}
First we recall some general notation from \cite{LT} and \cite{T}. Throughout this paper $G$ is a reductive group over an algebraically closed field $k$ of characteristic $p>0$, $T$ is a maximal torus of $G$ and $B^+$ is a Borel subgroup of $G$ containing $T$.
We denote the group of weights relative to $T$, i.e. the group of characters of $T$, by $X$. For $\lambda,\mu\in X$ we write $\mu\le\lambda$ if $\lambda-\mu$ is a sum of positive roots (relative to $B^+$). The Weyl group of $G$ relative to $T$ is denoted by $W$ and the set of dominant weights relative to $B^+$ is denoted by $X^+$. In the category of (rational) $G$-modules, i.e. $k[G]$-comodules, there are several special families of modules. For $\lambda\in X^+$ we have the irreducible $L(\lambda)$ of highest weight $\lambda$, and the induced module $\nabla(\lambda)=\ind_B^Gk_\lambda$, where $B$ is the opposite Borel subgroup to $B^+$ and $k_\lambda$ is the 1-dimensional $B$-module afforded by $\lambda$. The Weyl module and indecomposable tilting module associated to $\lambda$ are denoted by $\Delta(\lambda)$ and $T(\lambda)$. To each $G$-module $M$ we can associate its formal character ${\rm ch}\,M=\sum_{\lambda\in X}\dim M_\lambda e(\lambda)\in(\mb Z X)^W$, where $M_\lambda$ is the weight space associated to $\lambda$ and $e(\lambda)$ is the basis element corresponding to $\lambda$ of the group algebra $\mb Z X$ of $X$ over $\mb Z$. Composition and good or Weyl filtration multiplicities are denoted by $[M:L(\lambda)]$ and $(M:\nabla(\lambda))$ or $(M:\Delta(\lambda))$. For a weight $\lambda$, the character $\chi(\lambda)$ is given by Weyl's character formula \cite[II.5.10]{Jan}. If $\lambda$ is dominant, then ${\rm ch}\,\nabla(\lambda)={\rm ch}\,\Delta(\lambda)=\chi(\lambda)$. The $\chi(\lambda)$, $\lambda\in X^+$, form a $\mb Z$-basis of $(\mb Z X)^W$.
For $\alpha$ a root and $l\in\mb Z$, let $s_{\alpha,l}$ be the affine reflection of $\mb R\ot_{\mb Z}X$ defined by $s_{\alpha,l}(x)=x-a\alpha$, where $a=\la x,\alpha^\vee\ra-lp$. Mostly we replace $\la-,-\ra$ by a $W$-invariant inner product and then the cocharacter group of $T$ is identified with a lattice in $\mb R\ot_{\mb Z}X$ and $\alpha^\vee=\frac{2}{\la\alpha,\alpha\ra}\alpha$.
We have $s_{-\alpha,l}=s_{\alpha,-l}$ and the affine Weyl group $W_p$ is generated by the $s_{\alpha,l}$.
Choose $\rho\in\mathbb Q\ot_{\mb Z}X$ with $\la\rho,\alpha^\vee\ra=1$ for all $\alpha$ simple and define the dot action of $W_p$ on $\mb R\ot_{\mb Z}X$ by $w\cdot x=w(\lambda+\rho)-\rho$. The lattice $X$ is stable under the dot action.
The \emph{linkage principle} \cite[II.6.17,7.2]{Jan} says that if $L(\lambda)$ and $L(\mu)$ belong to the same $G$-block, then $\lambda$ and $\mu$ are $W_p$-conjugate under the dot action. We refer to \cite{Jan} part II for more details.

Unless stated otherwise, $G$ will be the general linear group $\GL_n$ or the symplectic group $\Sp_n$, $n=2m$, given by $\Sp_n=\{A\in\GL_n\,|\,A^TJA=J\}$, where $J=\begin{bmatrix}0&I_m\\-I_m&0\end{bmatrix}$ and $A^T$ is the transpose of $A$. The natural $G$-module $k^n$ is denoted by $V$.
Partitions with parts $<10$ may be written in ``exponential form": $(5,5,4,3,2)$ is denoted by $(5^2432)$, where we sometimes omit the brackets.

First assume $G=\GL_n$. We let $T$ be the group of diagonal matrices in $\GL_n$. Then $X$ is naturally identified with $\mb Z^n$ such that the $i$-th diagonal coordinate function corresponds to the $i$-th standard basis element $\ve_i$ of $\mb Z^n$. We let $B^+$ be the Borel subgroup of invertible upper triangular matrices corresponding to the set of positive roots $\ve_i-\ve_j$, $1\le i<j\le n$. Then a weight in $\mb Z^n$ is dominant if and only if it is weakly decreasing.
Such a weight $\lambda$ can uniquely be written as
$$[\lambda^1,\lambda^2]\stackrel{\rm def}{=}(\lambda^1_1,\lambda^1_2,\ldots,0,\ldots,0,\ldots,-\lambda^2_2,-\lambda^2_1)$$
where $\lambda^1=(\lambda^1_1,\lambda^1_2,\ldots)$ and $\lambda^2=(\lambda^2_1,\lambda^2_2,\ldots)$ are partitions with $l(\lambda^1)+l(\lambda^2)\le n$. Here $l(\xi)$ denotes the length of a partition $\xi$. So $X^+$ can be identified with pairs of partitions $(\lambda^1,\lambda^2)$ with $l(\lambda^1)+l(\lambda^2)\le n$.
We will also identify partitions with the corresponding Young diagrams. For $s_1,s_2\in\{1,\ldots,n\}$ with $s_1+s_2\le n$ we denote the subgroup of $W_p$ generated by the $s_{\alpha,l}$, $\alpha=\ve_i-\ve_j$, $i,j\in\{1,\ldots,s_1,n-s_2+1,\dots,n\}$ by $W_p^{s_1,s_2}$. This is the affine Weyl group of a root system of type $A_{s_1+s_2-1}$. The group $W$ acts on $\mb Z^n$ by permutations, and $W_p\cong W\ltimes pX_0$, where $X_0=\{\lambda\in\mb Z^n\,|\,|\lambda|=0\}$ is the type $A_{n-1}$ root lattice and $|\lambda|=\sum_{i=1}^n\lambda_i$. Note that $W_p^{s_1,s_2}\cong W^{s_1,s_2}\ltimes pX_0^{s_1,s_2}$, where $X_0^{s_1,s_2}$ consists of the vectors in $X_0$ which are $0$ at the positions in $\{s_1+1,\ldots,n-s_2\}$, and $W^{s_1,s_2}=\Sym(\{1,\ldots,s_1,n-s_2+1,\dots,n\})$. We will work with $$\rho=(n,n-1,\ldots,1)\,.$$
It is easy to see that if $\lambda,\mu\in X$ are $W_p$-conjugate and equal at the positions in $\{s_1+1,\ldots,n-s_2\}$, then they are $W_p^{s_1,s_2}$-conjugate. The same applies for the dot action.

Now assume $G=\Sp_n$. We let $T$ be the group of diagonal matrices in $\Sp_n$, i.e. the matrices $diag(d_1,\ldots,d_n)$ with $d_id_{i+m}=1$ for all $i\in\{1,\ldots,m\}$. Now $X$ is naturally identified with $\mb Z^m$ such that the $i$-th diagonal coordinate function corresponds to the $i$-th standard basis element $\ve_i$ of $\mb Z^m$. We let $B^+$ be the Borel subgroup corresponding to the set of positive roots $\ve_i\pm\ve_j$, $1\le i<j\le m$, $2\ve_i$, $1\le i\le m$.
We can now identify the dominant weights with $m$-tuples $(\lambda_1,\ldots,\lambda_m)$ with $\lambda_1\ge\lambda_2\ge\cdots\ge\lambda_m\ge0$, or with partitions $\lambda$ with $l(\lambda)\le m$. We denote the subgroup of $W_p$ generated by the $s_{\alpha,l}$, $\alpha=\ve_i\pm\ve_j$, $1\le i<j\le s$ or $\alpha=2\ve_i$, $1\le i\le s$ by $W_p(C_s)$ and we denote the subgroup of $W_p$ generated by the $s_{\alpha,l}$, $\alpha=\ve_i\pm\ve_j$, $1\le i<j\le s$ by $W_p(D_s)$. The group $W$ acts on $\mb Z^m$ by permutations and sign changes, and $W_p\cong W\ltimes pX_{ev}$, where $X_{ev}=\{\lambda\in\mb Z^m\,|\,|\lambda|\text{ even}\}$ is the type $C_m$ root lattice. Note that $W_p(C_s)\cong W(C_s)\ltimes pX_{ev}(C_s)$ and $W_p(D_s)\cong W(D_s)\ltimes pX_{ev}(C_s)$, where $X_{ev}(C_s)$ consists of the vectors in $X_{ev}$ which are $0$ at the positions $>s$, $W(C_s)$ is generated by the $s_\alpha=s_{\alpha,0}$, $\alpha=\ve_i\pm\ve_j$, $1\le i<j\le s$ or $\alpha=2\ve_i$, $1\le i\le s$, and $W(D_s)$ is generated by the $s_\alpha$, $\alpha=\ve_i\pm\ve_j$, $1\le i<j\le s$. The group $W(D_s)$ acts by permutations and an even number of sign changes. We have $$\rho=(m,m-1,\ldots,1)\,.$$
It is easy to see that if $\lambda,\mu\in X$ are $W_p$-conjugate and equal at the positions $>s$, then they are $W_p(C_s)$-conjugate. The same applies for the dot action.

In Section~3 of \cite{LT} and \cite{T} the Jantzen sum formula is studied under certain assumptions and this leads to a \emph{reduced Jantzen sum formula}. From this a partial order $\preceq$ on $X^+$ is deduced which is the reflexive, transitive closure of the order ``$\chi(w\cdot\mu)$ occurs for some $w\in W$ in the RHS of the reduced Jantzen sum formula associated to $\lambda$".
We now give some precise definitions. First assume $G=\GL_n$. Then we define

\begin{defgl}
$\mu\preceq\lambda$ if and only if there is a sequence of dominant weights $\lambda=\chi_1,\ldots,\chi_t=\mu$, $t\ge1$, such that for all $r\in\{1,\ldots,t-1\}$, $\chi_{r+1}=ws_{\alpha,l}\cdot\chi_r$ for some $w\in W^{l(\chi_r^1),l(\chi_r^2)}$, $\alpha=\ve_i-\ve_j$, $1\le i\le l(\chi_r^1)$, $n-l(\chi_r^2)<j\le n$, and $l\ge1$ with $\la\chi_r+\rho,\alpha^\vee\ra-lp\ge1$, and $\chi(s_{\alpha,l}\cdot\chi_r)\ne0$.
\end{defgl}

\noindent We put $$\Lambda_p=\{\lambda\in X^+\,|\,\lambda^h_1+l(\lambda^h)\le p\text{\ for all\ }h\in\{1,2\}\}$$
We will assume $s=(s_1,s_2)$ where $s_1,s_2\in\{1,\ldots,\min(n,p)\}$ with $s_1+s_2\le n$ and we put
$$\Lambda(s)=\Lambda(s_1,s_2)=\{\lambda\in X^+\,|\,l(\lambda^h)\le s_h\le p-\lambda^h_1\text{\ for all\ }h\in\{1,2\}\}$$
and $H=W_p$.

Next assume $G=\Sp_n$, $n=2m$. Then we define

\begin{defgl}
$\mu\preceq\lambda$ if and only if there is a sequence of dominant weights $\lambda=\chi_1,\ldots,\chi_t=\mu$, $t\ge1$, such that for all $r\in\{1,\ldots,t-1\}$, $\chi_{r+1}=ws_{\alpha,l}\cdot\chi_r$ for some $w\in\Sym(\{1,\ldots,l(\chi_r)\})$, $\alpha=\ve_i+\ve_j$, $1\le i<j\le l(\chi_r)$, and $l\ge1$ with $\la\chi_r+\rho,\alpha^\vee\ra-lp\ge1$, and all entries of $s_{\alpha,l}(\chi_r+\rho)$ distinct and strictly positive.
\end{defgl}

\noindent We put
$$\Lambda_p=\{\lambda\in X^+\,|\,\lambda_1+l(\lambda)\le p\}$$

\noindent Now we will assume $s\in\{1,\ldots,\min(m,p)\}$ and we put
$$\Lambda(s)=\{\lambda\in X^+\,|\,l(\lambda)\le s\le p-\lambda_1\}$$
and $H=W_p(D_s)$.

We return to the general case $G=\GL_n$ or $G=\Sp_n$. For a subset $\Lambda$ of $X^+$ and a $G$-module $M$ we say that $M$ \emph{belongs to $\Lambda$} if all composition factors have highest weight in $\Lambda$ and we denote by $O_\Lambda(M)$ the largest submodule of $M$ which belongs to $\Lambda$. We denote the category of $G$-modules which belong to $\Lambda$ by $\mc C_\Lambda$. The category $\mc C_\Lambda$ is the module category of the algebra $O_\Lambda(k[G])^*$, see \cite[Ch A]{Jan} for the relevant definitions and explanation. Let $\Lambda\subseteq\Lambda(s)$ be $\preceq$-saturated. It was shown in Prop~3.1(ii) in \cite{T} and \cite{LT} that the algebra $O_\Lambda(k[G])^*$ is quasihereditary for the partial order $\preceq$ such that the irreducible, standard/costandard and tilting modules are the irreducible, Weyl/induced and tilting modules for $G$ with the same label.

\subsection{Arrow diagrams, and cap(-curl) diagrams and codiagrams}\label{ss.arrow_diagrams}\ \\
{\bf Arrow and cap(-curl) diagrams}. We now recall the definition of the arrow and cap(-curl) diagram from \cite[Sect~5]{LT} and \cite[Sect~5]{T} which is based on \cite{CdV} and \cite{Sh}.
Recall the definitions of $\preceq$, $\Lambda(s)$ and $H$ from Section~\ref{ss.reductive_groups}.
First we assume $G=\GL_n$. An arrow diagram has $p$ \emph{nodes} on a (horizontal) \emph{line} with $p$ \emph{labels}: $0,\ldots,p-1$. The $i$-th node from the left has label $i-1$.
Although $0$ and $p-1$ are not connected we consider them as neighbours and we will identify a diagram with any of its cyclic shifts. So when we are going to the left through the nodes we get $p-1$ after $0$ and when we are going to the right we get $0$ after $p-1$.
Next we choose $s_1,s_2\in\{1,\ldots,\min(n,p)\}$ with $s_1+s_2\le n$ and put a wall below the line between $\rho_{s_1}$ and $\rho_{s_1}-1$ mod $p$, and a wall above the line between $\rho_{s_2'}=s_2$ and $s_2+1$ mod $p$. Then we can also put in a top and bottom \emph{value} for each label. A value and its corresponding label are always equal mod $p$. Below the line we start with $\rho_{s_1}$ immediately to the right of the wall, and then increasing in steps of $1$ going to the right: $\rho_{s_1},\rho_{s_1}+1,\ldots,\rho_{s_1}+p-1$. Above the line we start with $\rho_{s_2'}=s_2$ immediately to the left of the wall, and then decreasing in steps of $1$ going to the left: $s_2,s_2-1,\ldots,s_2-p+1$.
For example, when $p=5$, $n=5$ and $s_1=s_2=1$, then $\rho_{s_1}=s_1'=5$, $\rho_{s_2'}=s_2=1$ and we have labels
$$
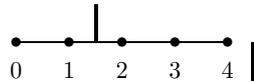
\resizebox{3.2cm}{.5cm}{\xy
(0,0)="a1"*{\bullet};(0,-4.5)*{0};
(8.5,0)="a2"*{\bullet};(8.5,-4.5)*{1};(12.8,3)*{\rule[0mm]{.3mm}{6mm}};
(17,0)="a3"*{\bullet};(17,-4.5)*{2};
(25.5,0)="a4"*{\bullet};(25.5,-4.5)*{3};
(34,0)="a5"*{\bullet};(34,-4.5)*{4};(38,-3)*{\rule[0mm]{.3mm}{6mm}};
{"a1";"a5"**@{-}}; 
\endxy}\ \ \,
$$
and values
$$
\resizebox{3.2cm}{.5cm}{\xy
(0,0)="a1"*{\bullet};(0,4.5)*{0};(0,-4.5)*{5};
(8.5,0)="a2"*{\bullet};(8.5,4.5)*{1};(8.5,-4.5)*{6};(12.8,3)*{\rule[0mm]{.3mm}{6mm}};
(17,0)="a3"*{\bullet};(17,4.5)*{-3};(17,-4.5)*{7};
(25.5,0)="a4"*{\bullet};(25.5,4.5)*{-2};(25.5,-4.5)*{8};
(34,0)="a5"*{\bullet};(34,4.5)*{-1};(34,-4.5)*{9};(38,-3)*{\rule[0mm]{.3mm}{6mm}};
{"a1";"a5"**@{-}}; 
\endxy}\ \ .
$$
For $\lambda=[\lambda^1,\lambda^2]\in\Lambda(s)=\Lambda(s_1,s_2)$ we now form the ($(s_1,s_2)$-)\emph{arrow diagram} by putting $s_1$ arrows below the line ($\land$) that point \emph{from} the values $(\rho+\lambda)_1,\ldots,$ $(\rho+\lambda)_{s_1}$, i.e. $\rho_1+\lambda^1_1,\ldots,\rho_{s_1}+\lambda^1_{s_1}$, or from the corresponding labels, and $s_2$ arrows above the line ($\vee$) that point from the values $(\rho+\lambda)_{1'},\ldots,(\rho+\lambda)_{s_2'}$, i.e. $1-\lambda^2_1,\ldots,s_2-\lambda^2_{s_2}$, or to the corresponding labels.
So in the above example the arrow diagram of $\lambda=[4,4]$ is
$$
\resizebox{3.2cm}{.5cm}{\xy
(0,0)="a1"*{\bullet};(0,-4.5)*{0};
(8.5,0)="a2"*{\bullet};(8.5,-4.5)*{1};(12.8,3)*{\rule[0mm]{.3mm}{6mm}};
(17,0)="a3"*{\bullet};(17,-4.5)*{2};(17,1.5)*{\vee};
(25.5,0)="a4"*{\bullet};(25.5,-4.5)*{3};
(34,0)="a5"*{\bullet};(34,-4.5)*{4};(34,-1.5)*{\land};(38,-3)*{\rule[0mm]{.3mm}{6mm}};
{"a1";"a5"**@{-}}; 
\endxy}\ \ .
$$
In such a diagram we frequently omit the nodes and/or the labels. When it has already been made clear what the labels are and where the walls are, we can simply represent the arrow diagram by a string of single arrows ($\land$, $\vee$), opposite pairs of arrows ($\times$) and symbols ${\rm o}$ to indicate the absence of an arrow. In the above example $\lambda=[4,4]$ is then represented by ${\rm oo}\nts\vee\nts{\rm o}\land$ and $\lambda=[2,4]$ is represented by ${\rm oo}\nts\times\nts{\rm oo}$.

We can form the arrow diagram of $\lambda$ as follows. First line up $s_1$ arrows immediately to the right of the wall below the line and then move them to the right to the correct positions. The arrow furthest from the wall corresponds to $\lambda^1_1$, and the arrow closest to the wall corresponds to $\lambda^1_{s_1}$. Then line up $s_2$ arrows immediately to the left of the wall above the line and then move them to the left to the correct positions. The arrow furthest from the wall corresponds to $\lambda^2_1$, and the arrow closest to the wall corresponds to $\lambda^2_{s_2}$.
The part of $\lambda^1$ corresponding to an arrow below the line equals the number of nodes without a $\land$ from that arrow to the wall going to the left and the part of $\lambda^2$ corresponding to an arrow below the line equals the number of nodes without a $\vee$ from that arrow to the wall going to the right.

When we speak of ``arrow pairs", also in the $\Sp_n$-case below, it is understood that both arrows are single, i.e. neither of the two arrows is part of an $\times$. The arrows need not be consecutive in the diagram.
We now define the \emph{cap diagram} $c_\lambda$ of the arrow diagram associated to $\lambda$ as follows. We assume that the arrow diagram is cyclically shifted such that at least one of the walls is between the first and last node. We select one such wall and when we speak of ``the wall" it will be the other wall. All caps are anti-clockwise, starting from the rightmost node. We start on the left side of the wall. We form the caps recursively. Find an arrow pair $\vee\land$ that are neighbours in the sense that the only arrows in between are already connected with a cap or are part of an $\times$, and connect them with a cap. Repeat this until there are no more such arrow pairs. Now the unconnected arrows that are not part of an $\times$ form a sequence $\land\cdots\land\vee\cdots\vee$. Note that none of these arrows occur inside a cap. The caps on the right side of the wall are formed in the same way.
For example, when $p=17$, $n=20$, $s_1=8$, $s_2=7$ and $\lambda=[9654^22,8^243^22]$, then
$c_\lambda$ is
$$
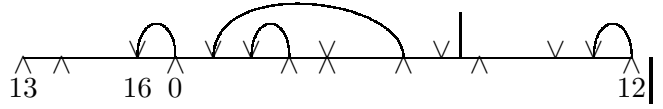
\xy
(0,0)="a1";(0,-1)*{\land};(0,-4)*{13};
(5,0)="a2";(5,-1)*{\land};
(10,0)="a3";
(15,0)="a4";(15,1)*{\vee};(15,-4)*{16};
(20,0)="a5";(20,-1)*{\land};(20,-4)*{0};
(25,0)="a6";(25,1)*{\vee};
(30,0)="a7";(30,1)*{\vee};
(35,0)="a8";(35,-1)*{\land};
(40,0)="a9";(40,-1)*{\land};(40,1)*{\vee};
(45,0)="a10";
(50,0)="a11";(50,-1)*{\land};
(55,0)="a12";(55,1)*{\vee};(57.5,3)*{\rule[0mm]{.3mm}{6mm}};
(60,0)="a13";(60,-1)*{\land};
(65,0)="a14";
(70,0)="a15";(70,1)*{\vee};
(75,0)="a16";(75,1)*{\vee};
(80,0)="a17";(80,-1)*{\land};(80,-4)*{12};(82.5,-3)*{\rule[0mm]{.3mm}{6mm}};
{"a1";"a17"**@{-}}; 
"a5";"a4"**\crv{(20,6)&(15,6)};
"a8";"a7"**\crv{(35,6)&(30,6)};
"a11";"a6"**\crv{(50,8)&(25,11)};
"a17";"a16"**\crv{(80,6)&(75,6)};
\endxy\ .$$
Note that the nodes with labels $5,9,15$ have no arrow.

Now assume $G=\Sp_n$. An arrow diagram has $(p+1)/2$ \emph{nodes} on a (horizontal) \emph{line} with $p$ \emph{labels}: $0$ and $\pm i$, $i\in\{1,\ldots,(p-1)/2\}$. The $i$-th node from the left has top label $-(i-1)$ and a bottom label $i-1$. So the first node is the only node whose top and bottom label are the same. Next we choose $s\in\{1,\ldots,\min(m,p)\}$ and put a wall between $\rho_s$ and $\rho_s-1$ mod $p$. So when $\rho_s=(p+1)/2$ mod $p$ we can put the wall above or below the line, otherwise there is only one possibility.
Then we can also put in the \emph{values}, one for each label. A value and its corresponding label are always equal mod $p$. We start with $\rho_s$ immediately after the wall in the anti-clockwise direction, and then increasing in steps of $1$ going in the anti-clockwise direction around the line: $\rho_s,\rho_s+1,\ldots,\rho_s+p-1$.
For example, when $p=5$, $m=7$ and $s=2$, then $\rho_s=6$ and we have labels
$$
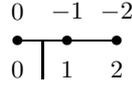
\resizebox{1.6cm}{.5cm}{\xy
(0,0)="a1"*{\bullet};(0,4.5)*{0};(0,-4.5)*{0};(3.8,-3)*{\rule[0mm]{.3mm}{6mm}};
(7.5,0)="a2"*{\bullet};(7.5,4.5)*{-1};(7.5,-4.5)*{1};
(15,0)="a3"*{\bullet};(15,4.5)*{-2};(15,-4.5)*{2};
{"a1";"a3"**@{-}}; 
\endxy}
$$
(usually we omit the top labels), and values
$$
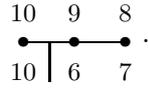
\resizebox{1.6cm}{.5cm}{\xy
(0,0)="a1"*{\bullet};(0,4.5)*{10};(0,-4.5)*{10};(3.8,-3)*{\rule[0mm]{.3mm}{6mm}};
(7.5,0)="a2"*{\bullet};(7.5,4.5)*{9};(7.5,-4.5)*{6};
(15,0)="a3"*{\bullet};(15,4.5)*{8};(15,-4.5)*{7};
{"a1";"a3"**@{-}}; 
\endxy}\ .
$$
For a partition $\lambda\in\Lambda(s)$ we now form the ($s$-)\emph{arrow diagram} by putting in $s$ arrows ($\vee$ or $\land$) that point \emph{from} the values $(\rho+\lambda)_1,\ldots,(\rho+\lambda)_s$, or the corresponding labels. In case of the label $0$ we have two choices for the arrow. So in the above example the arrow diagram of $\lambda=(1^2)$ is
$$
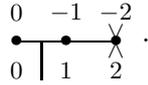
\resizebox{1.6cm}{.5cm}{\xy
(0,0)="a1"*{\bullet};(0,4.5)*{0};(0,-4.5)*{0};(3.8,-3)*{\rule[0mm]{.3mm}{6mm}};
(7.5,0)="a2"*{\bullet};(7.5,4.5)*{-1};(7.5,-4.5)*{1};
(15,0)="a3"*{\bullet};(15,1.5)*{\vee};(15,-1.5)*{\land};(15,4.5)*{-2};(15,-4.5)*{2};
{"a1";"a3"**@{-}}; 
\endxy}\ .
$$
As in the $\GL_n$-case we can simply represent the arrow diagram by a string of single arrows ($\land$, $\vee$), opposite pairs of arrows ($\times$) and symbols ${\rm o}$ to indicate the absence of an arrow. In the above example $\lambda=(1^2)$ is then represented by ${\rm oo}\times$ and $\lambda=(32)$ is represented by $\vee{\rm o}\vee$ or $\land{\rm o}\vee$.

We can form the arrow diagram of $\lambda$ by first lining all $s$ arrows up against the wall and then moving them in the anticlockwise direction to the right positions. The arrow furthest from the wall (in the anti-clockwise direction) corresponds to $\lambda_1$, and the arrow closest to the wall corresponds to $\lambda_s$. The part corresponding to an arrow equals the number of labels without an arrow from that arrow to the wall in the clockwise direction.

We now define the \emph{cap-curl diagram} $c_\lambda$ of the arrow diagram associated to $\lambda$ as follows. All caps and curls are anti-clockwise, starting from the arrow closest to the wall. We start on the left side of the wall. We first form the caps recursively. Find an arrow pair $\vee\land$ that are neighbours in the sense that the only arrows in between are already connected with a cap or are part of an $\times$, and connect them with a cap. Repeat this until there are no more such arrow pairs.
Now the unconnected arrows that are not part of an $\times$ form a sequence $\land\cdots\land\vee\cdots\vee$. We connect consecutive (in the mentioned sequence) $\land\land$ pairs with a curl, starting from the left. At the end the unconnected arrows that are not part of an $\times$ form a sequence $\land\vee\cdots\vee$ or just a sequence of $\vee$'s. Note that none of these arrows occur inside a cap or curl.
The caps on the right side of the wall are formed in the same way. The curls now connect consecutive $\vee\vee$ pairs and are formed starting from the right. So at the end the unconnected arrows that are not part of an $\times$ form a sequence $\land\cdots\land\vee$ or just a sequence of $\land$'s. Again, none of these arrows occur inside a cap or curl.
For example, when $p=23$, $m=17$, $s=12$ and $\lambda=(11,11,11,11,11,11,10,6,4,4,1)$, then
$c_\lambda$ is
$$
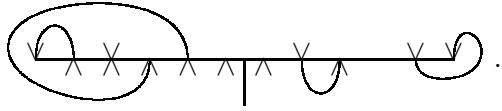
\xy
(0,0)="a1";(0,1)*{\vee};
(5,0)="a2";(5,-1)*{\land};
(10,0)="a3";(10,1)*{\vee};(10,-1)*{\land};
(15,0)="a4";(15,-1)*{\land};
(20,0)="a5";(20,-1)*{\land};
(25,0)="a6";(25,-1)*{\land};(27.5,-3)*{\rule[0mm]{.3mm}{6mm}};
(30,0)="a7";(30,-1)*{\land};
(35,0)="a8";(35,1)*{\vee};
(40,0)="a9";(40,-1)*{\land};
(45,0)="a10";
(50,0)="a11";(50,1)*{\vee};
(55,0)="a12";(55,1)*{\vee};
{"a1";"a12"**@{-}}; 
"a2";"a1"**\crv{(5,6)&(0,6)}; 
"a9";"a8"**\crv{(40,-6)&(35,-6)}; 
"a5";"a4"**\crv{(20,8)&(-5,9)&(-5,-5)&(15,-7)}; 
"a11";"a12"**\crv{(50,-3)&(59,-3)&(59,4)&(55,4)}; 
\endxy\ .$$
Note that the $10$-th node which has labels $\pm9$ and values $9$ and $14$, has no arrow.

\medskip

{\bf $(\land,\vee)$-sequences and length functions}. We now return to the general case $G=\GL_n$ or $G=\Sp_n$. First we introduce some combinatorial tools to express the order $\preceq$ in terms of arrow diagrams. This is based on the treatment in \cite[Sect 5]{BS} and \cite[Sect 8]{CdV}. Let $\xi,\eta$ be sequences with values in $\{\land,\vee\}$. We say that $\xi$ and $\eta$ are \emph{conjugate} if they have the same length and the same number of $\land$'s$\mod 2$. We say they are \emph{strongly conjugate} if they have the same length and the same number of $\land$'s.

\begin{defgl}
We write $\xi\preceq\eta$ if $\xi$ can be obtained from $\eta$ by repeatedly replacing an arrow pair $\vee\land$ or an arrow pair $\land\land$ by the opposite arrow pair.
\end{defgl}

Clearly, $\xi\preceq\eta$ implies that $\xi$ and $\eta$ are conjugate. If $\xi$ and $\eta$ are strongly conjugate and $\xi\preceq\eta$, then $\xi$ can be obtained from $\eta$ by repeatedly replacing an arrow pair $\vee\land$ by the opposite arrow pair. For $\eta,\xi\in\{\land,\vee\}^r$ and $i\in\{1,\ldots,r\}$ we define
$$l_i(\eta)=|\{j\in\{i,\ldots,r\}\,|\,\eta_j=\land\}|\text{\ and\ }l_i(\eta,\xi)=l_i(\eta)-l_i(\xi)\,.$$
\noindent Note that $l_i(\eta,\xi)$ equals
$$|\{j\in\{i,\ldots,r\}\,|\,\eta_j\ne\xi_j\text{\ and }\eta_j=\land\}|-|\{j\in\{i,\ldots,r\}\,|\,\eta_j\ne\xi_j\text{\ and }\,\xi_j=\land\}|\,.$$
Then we have

\medskip
\emph{$\xi\preceq\eta\iff\xi$ and $\eta$ are conjugate and $l_i(\eta,\xi)\ge0$ for all $i\in\{2,\ldots,r\}$.}
\medskip

\noindent Put $l(\eta)=\sum_{i=2}^rl_i(\eta)$ and $l(\eta,\xi)=l(\eta)-l(\xi)=\sum_{i=2}^rl_i(\eta,\xi)$. Call replacing an arrow pair $\land\land$ in the first two positions or a consecutive arrow pair $\vee\land$ by the opposite arrow pair an \emph{elementary operation}.
If $\xi\preceq\eta$, then $l(\eta,\xi)$ is the minimal number of elementary operations needed to obtain $\xi$ from $\eta$.\\

For $\lambda\in\Lambda(s)$ we define the \emph{associated pair of $(\land,\vee)$-sequences} $(\eta^1,\eta^2)$ as follows. If $G=\GL_n$, then $\eta^1$ is the sequence of single arrows to the left of the wall in the (cyclically shifted) arrow diagram of $\lambda$, and $\eta^2$ is the sequence of single arrows to the right of the wall. This pair is well-defined up to order. If $G=\Sp_n$, then $\eta^1$ is the sequence of single arrows to the left of the wall in the arrow diagram of $\lambda$, and $\eta^2$ is the sequence of single arrows to the right of the wall, rotated 180 degrees. For example, when $G=\GL_n$ and the arrow diagram of $\lambda$ is
$$\xy
(0,0)="a1";(0,-1)*{\land};
(5,0)="a2";(5,-1)*{\land};
(10,0)="a3";
(15,0)="a4";(15,1)*{\vee};
(20,0)="a5";(20,-1)*{\land};
(25,0)="a6";(25,1)*{\vee};
(30,0)="a7";(30,1)*{\vee};
(35,0)="a8";(35,-1)*{\land};
(40,0)="a9";(40,-1)*{\land};(40,1)*{\vee};
(45,0)="a10";
(50,0)="a11";(50,-1)*{\land};
(55,0)="a12";(55,1)*{\vee};(57.5,3)*{\rule[0mm]{.3mm}{6mm}};
(60,0)="a13";(60,-1)*{\land};
(65,0)="a14";
(70,0)="a15";(70,1)*{\vee};
(75,0)="a16";(75,1)*{\vee};
(80,0)="a17";(80,-1)*{\land};(82.5,-3)*{\rule[0mm]{.3mm}{6mm}};
{"a1";"a17"**@{-}}; 
\endxy\ ,$$
then $(\eta^1,\eta^2)=(\land\land\vee\land\vee\vee\land\land\vee,\land\vee\vee\land)$, and
when $G=\Sp_n$ and the arrow diagram of $\lambda$ is
$$\xy
(0,0)="a1";(0,1)*{\vee};
(5,0)="a2";(5,-1)*{\land};
(10,0)="a3";(10,1)*{\vee};(10,-1)*{\land};
(15,0)="a4";(15,-1)*{\land};
(20,0)="a5";(20,-1)*{\land};
(25,0)="a6";(25,-1)*{\land};(27.5,-3)*{\rule[0mm]{.3mm}{6mm}};
(30,0)="a7";(30,-1)*{\land};
(35,0)="a8";(35,1)*{\vee};
(40,0)="a9";(40,-1)*{\land};
(45,0)="a10";
(50,0)="a11";(50,1)*{\vee};
(55,0)="a12";(55,1)*{\vee};
{"a1";"a12"**@{-}}; 
\endxy\ ,$$
then $(\eta^1,\eta^2)=(\vee\land\land\land\land,\land\land\vee\land\vee)$.
For $\lambda,\mu\in\Lambda(s)$ with associated pairs of $(\land,\vee)$-sequences $(\eta^1,\eta^2)$ and $(\xi^1,\xi^2)$ we put
$$n(\lambda)=l(\eta^1)+l(\eta^2)\text{\ and\ }n(\lambda,\mu)=l(\eta^1,\xi^1)+l(\eta^2,\xi^2)=n(\lambda)-n(\mu)\,.$$
Note that $n(\lambda,\mu)$ is independent of $s$.

If below $\lambda,\mu\in\Lambda(s)$, then we let $(\eta^1,\eta^2)$ and $(\xi^1,\xi^2)$ be the pairs of $(\land,\vee)$-sequences associated to $\lambda$ and $\mu$. If furthermore $G=\Sp_n$ and the arrow diagram of $\lambda$ has an arrow at $0$, then we assume that the parity of the number of $\land$'s in the arrow diagram of $\mu$ is the same as that for $\lambda$. This only requires a possible change of an arrow at $0$ to its opposite in the arrow diagram of $\lambda$. For $\lambda,\mu\in\Lambda(s)$ we have by \cite[Rem 5.1.1]{T} and \cite[Rem 5.1.1]{LT} that

\medskip
{\it \noindent $\lambda$ and $\mu$ are $H$-conjugate under the dot action if and only if
the arrow diagram of $\mu$ has its single arrows and its $\times$'s at the same nodes as the arrow diagram of $\lambda$ and
$\begin{cases}
\xi^i\text{\ and\ }\eta^i\text{\ are strongly conjugate for all\ } i\in\{1,2\}&\text{if\ } G=\GL_n,\\
\xi^i\text{\ and\ }\eta^i\text{\ are conjugate for all\ } i\in\{1,2\}&\text{if\ } G=\Sp_n.
\end{cases}$}
\medskip

\noindent Furthermore, for $\lambda\in\Lambda(s)$ and $\mu\in X^+$ we have 

\medskip
\centerline{\it $\mu\preceq\lambda\Leftrightarrow \mu\in\Lambda(s)\cap H\cdot\lambda\text{\ and\ \,} \xi^1\preceq\eta^1\text{\ and\ \,}\xi^2\preceq\eta^2$.}
\medskip

{\bf More cap(-curl) diagrams, and codiagrams}. We now recall from \cite[Sect~6,7]{LT} and \cite[Sect~6,7]{T} the definitions of cap(-curl) diagrams associated to two weights, and codiagrams. Let $\lambda,\mu\in\Lambda(s)$ with $\mu\preceq\lambda$. Then the arrow diagram of $\mu$ has its single arrows and its $\times$'s at the same nodes as the arrow diagram of $\lambda$. If $G=\Sp_n$ and the arrow diagram of $\lambda$ has an arrow at $0$, then we assume that the parity of the number of $\land$'s in the arrow diagram of $\mu$ is the same as that for $\lambda$. This only requires a possible change of an arrow at $0$ to its opposite in the arrow diagram of $\lambda$ (or $\mu$). If there is no arrow at $0$, then these parities will automatically be the same, since $\mu$ is $W_p(D_{l(\lambda)})$-conjugate to $\lambda$ under the dot action. The \emph{cap-curl diagram} $c_{\lambda\mu}$ associated to $\lambda$ \emph{and $\mu$} by replacing each arrow in $c_\lambda$ by the arrow from the arrow diagram of $\mu$ at the same node. Put differently, we put the caps and curls from $c_\lambda$ on top of the arrow diagram of $\mu$. We say that $c_{\lambda\mu}$ is \emph{oriented} if all caps and curls in $c_{\lambda\mu}$ are oriented (clockwise or anti-clockwise). It is not hard to show that when $c_{\lambda\mu}$ is oriented, the arrow diagrams of $\lambda$ and $\mu$ are the same at the nodes which are not endpoints of a cap or a curl in $c_\lambda$.

For example, when $G=\GL_n$, $p=5$, $n=7$, $s_1=2$, $s_2=3$ and $\lambda=[32,21^2]$. Then $\rho_{s_1}=s_1'=6$, and $c_\lambda$ (cyclically shifted) is
$$
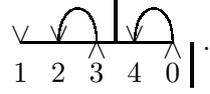
\xy
(0,0)="a1";(0,1)*{\vee};(0,-4)*{1};
(5,0)="a2";(5,1)*{\vee};(5,-4)*{2};
(10,0)="a3";(10,-1)*{\land};(10,-4)*{3};(12.5,3)*{\rule[0mm]{.3mm}{6mm}};
(15,0)="a4";(15,1)*{\vee};(15,-4)*{4};
(20,0)="a5";(20,-1)*{\land};(20,-4)*{0};(22.5,-3)*{\rule[0mm]{.3mm}{6mm}};
{"a1";"a5"**@{-}}; 
"a3";"a2"**\crv{(10,6)&(5,6)}; 
"a5";"a4"**\crv{(20,6)&(15,6)} 
\endxy\ .$$
The $\mu\in X^+$ with $\mu\prec\lambda$ are $[2^2,1^3],[31,21],[21,1^2],[3,2],[2,1]$,
with (cyclically shifted) arrow diagrams
$$\vee\vee\land\land\vee,\ \vee\land\vee\vee\land,\ \vee\land\vee\land\vee,\ \land\vee\vee\vee\land,\ \land\vee\vee\land\vee\,.$$
Only for the first three $c_{\lambda\mu}$ is oriented.

When $G=\Sp_n$, $p=11$, $m=7$, $s=5$ and $\lambda=(6^332)$. Then $\rho_s=3$ and $c_\lambda$ is
$$\xy
(0,0)="a1";(0,1)*{\vee};
(5,0)="a2";(5,-1)*{\land};
(10,0)="a3";(10,-1)*{\land};(12.5,-3)*{\rule[0mm]{.3mm}{6mm}};
(15,0)="a4";
(20,0)="a5";(20,1)*{\vee};
(25,0)="a6";(25,-1)*{\land};
{"a1";"a6"**@{-}}; 
"a2";"a1"**\crv{(5,6)&(0,6)}; 
"a6";"a5"**\crv{(25,-6)&(20,-6)} 
\endxy\ .$$
The $\mu\in X^+$ with $\mu\prec\lambda$ are $(6^321), (65^232), (65^221), (5^2432), (5^2421), (4^332),$\\ $(4^321)$,
with arrow diagrams
\begin{align*}
&\vee\land\land\,{\rm o}\land\vee,\ \land\vee\land\,{\rm o}\vee\land,\ \land\vee\land\,{\rm o}\land\vee,\ \land\land\vee\,{\rm o}\vee\land,\\
&\land\land\vee\,{\rm o}\land\vee,\ \vee\vee\vee\,{\rm o}\vee\land,\ \vee\vee\vee\,{\rm o}\land\vee\,.
\end{align*}
Only for the first three $c_{\lambda\mu}$ is oriented.

Finally, we define \emph{cap} or \emph{cap-curl codiagram} $co_\mu$ of the arrow diagram associated to $\mu\in\Lambda(s)$ by reversing the roles of $\land$ and $\vee$ in the definition of $c_\lambda$. So all caps and curls in $co_\mu$ are clockwise.
In the case $G=\Sp_n$ the caps now have their curve below the line when they are to the left of the wall and above the line when they are to the right of the wall. If $\mu,\lambda\in\Lambda(s)$ with $\mu\preceq\lambda$, then  we define cap or cap-curl codiagram $co_{\mu\lambda}$ associated to $\mu$ \emph{and $\lambda$} by replacing each arrow in $co_\mu$ by the arrow from the arrow diagram of $\lambda$ at the same node.\footnote{Again we assume that if $G=\Sp_n$ and the arrow diagram of $\lambda$ has an arrow at $0$, then the parity of the number of $\land$'s in the arrow diagram of $\lambda$ is the same as that for $\mu$.} We say that $co_{\mu\lambda}$ is \emph{oriented} if all caps and curls in $co_{\mu\lambda}$ are oriented (clockwise or anti-clockwise). We refer to \cite[Sect~7]{T} and \cite[Sect~7]{LT} for more details and just give two examples from these papers.

When $G=\Sp_n$, $p=11$, $m=7$, $s=5$ and $\mu=(4^321)$. Then $\rho_s=3$ and $co_\mu$ is
$$
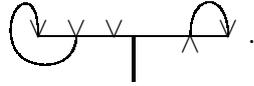
\xy
(0,0)="a1";(0,1)*{\vee};
(5,0)="a2";(5,1)*{\vee};
(10,0)="a3";(10,1)*{\vee};(12.5,-3)*{\rule[0mm]{.3mm}{6mm}};
(15,0)="a4";
(20,0)="a5";(20,-1)*{\land};
(25,0)="a6";(25,1)*{\vee};
{"a1";"a6"**@{-}}; 
"a2";"a1"**\crv{(5,-4)&(-4,-5)&(-4,5)&(0,5)}; 
"a6";"a5"**\crv{(25,6)&(20,6)} 
\endxy\ .$$
Consider two dominant weights $\lambda$ with $\mu\preceq\lambda$: $(6^332)$ and $(5^2432)$
with arrow diagrams $\vee\land\land{\rm o}\vee\land$ and $\land\land\vee{\rm o}\vee\land$.
Only for the last $co_{\mu\lambda}$ is oriented.

When $G=\GL_n$, $p=5$, $n=7$, $s_1=2$, $s_2=3$ and $\mu=[2,1]$. Then $\rho_{s_1}=s_1'=6$, and $co_\mu$ (cyclically shifted) is
$$
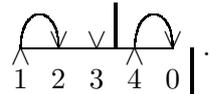
\xy
(0,0)="a1";(0,-1)*{\land};(0,-4)*{1};
(5,0)="a2";(5,1)*{\vee};(5,-4)*{2};
(10,0)="a3";(10,1)*{\vee};(10,-4)*{3};(12.5,3)*{\rule[0mm]{.3mm}{6mm}};
(15,0)="a4";(15,-1)*{\land};(15,-4)*{4};
(20,0)="a5";(20,1)*{\vee};(20,-4)*{0};(22.5,-3)*{\rule[0mm]{.3mm}{6mm}};
{"a1";"a5"**@{-}}; 
"a2";"a1"**\crv{(5,6)&(0,6)}; 
"a5";"a4"**\crv{(20,6)&(15,6)} 
\endxy\ .$$
Consider two dominant weights $\lambda$ with $\mu\preceq\lambda$: $[31,21]$ and $[32,21^2]$
with (cyclically shifted) arrow diagrams $\vee\land\vee\vee\land$ and $\vee\vee\land\vee\land$.
Only for the first $co_{\mu\lambda}$ is oriented.

\subsection{The translation functors}\label{ss.translation}
We recall from \cite[Sect~4]{LT} and \cite[Sect~4]{T} the definition and basic properties of certain translation functors and in the case of $G=\Sp_n$ we will also introduce certain refined translation functors.
For simplicity we do not quite state things in the same generality as in \cite{LT} and \cite{T}: we work below with the set $\Lambda(s)$ rather than the set $\Lambda_s$ from \cite{LT} or the set $\Lambda_p$ as in \cite{T}.
For $\lambda\in X^+$ the projection functor ${\rm pr}_\lambda:\{G\text{-modules}\}\to\{G\text{-modules}\}$ is defined by ${\rm pr}_\lambda M=O_{W_p\cdot\lambda\cap X^+}(M)$. Then $M=\bigoplus_\lambda{\rm pr}_\lambda M$ where the sum is over a set of representatives of the $W_p$-linkage classes in $X^+$, see \cite[II.7.3]{Jan}. Recall the definitions of $\preceq$, $\Lambda(s)$ and $H$ from Section~\ref{ss.reductive_groups}.

First assume $G=\GL_n$. For $\lambda=[\lambda^1,\lambda^2]\in X^+$, let ${\rm Supp}_1(\lambda)$ be the set of all $\mu=[\mu^1,\mu^2]\in X^+$ which can be obtained by adding a box to $\lambda^1$ or removing a box from $\lambda^2$, but not both, and let ${\rm Supp}_2(\lambda)$ be the set of all $\mu=[\mu^1,\mu^2]\in X^+$ which can be obtained by removing a box from $\lambda^1$ or adding a box to $\lambda^2$, but not both. Now let $\lambda,\lambda'\in X^+$ with $\lambda'\in{\rm Supp}_h(\lambda)$, $h\in\{1,2\}$. Then we have for the \emph{translation functor}, see \cite[II.7.6]{Jan}, $T_\lambda^{\lambda'}:\{G\text{-modules}\}\to\{G\text{-modules}\}$ that $T_\lambda^{\lambda'}M={\rm pr}_{\lambda'}(({\rm pr}_\lambda M)\ot V)$ when $h=1$ and $T_\lambda^{\lambda'}M={\rm pr}_{\lambda'}(({\rm pr}_\lambda M)\ot V^*)$ when $h=2$.
Furthermore, $T_\lambda^{\lambda'}$ is exact and left and right adjoint to $T_{\lambda'}^\lambda$.
An application of Brauer's formula shows that, for $\lambda'\in{\rm Supp}_h(\lambda)$, $h\in\{1,2\}$, and $\mu\in X^+\cap  W_p\cdot\lambda$, $T_\lambda^{\lambda'}\nabla(\mu)$ has a good filtration with sections $\nabla(\nu)$, $\nu\in{\rm Supp}_h(\mu)\cap  W_p\cdot\lambda'$, and the analogue for Weyl modules and Weyl filtrations also holds. We refer to \cite[Sect~4]{T} for further explanation.
To unify notation with the case $G=\Sp_n$, which we will discuss next, we put $\widetilde T_\lambda^{\lambda'}=T_\lambda^{\lambda'}$.

Now assume $G=\Sp_n$. For $\lambda\in X^+$, let ${\rm Supp}(\lambda)$ be the set of all partitions of length $\le m$ which can be obtained by adding a box to $\lambda$ or removing a box from $\lambda$. Then we have for the \emph{translation functor} $T_\lambda^{\lambda'}:\{G\text{-modules}\}\to\{G\text{-modules}\}$ that $T_\lambda^{\lambda'}M={\rm pr}_{\lambda'}(({\rm pr}_\lambda M)\ot V)$. Furthermore, $T_\lambda^{\lambda'}$ is exact and left and right adjoint to $T_{\lambda'}^\lambda$.
Note that, for $\mu\in X^+\cap  W_p\cdot\lambda$, $T_\lambda^{\lambda'}\nabla(\mu)$ has a good filtration with sections $\nabla(\nu)$, $\nu\in{\rm Supp}(\mu)\cap  W_p\cdot\lambda'$, and the analogue for Weyl modules and Weyl filtrations also holds.

We now define certain refined translation functors.
If $\Lambda\subseteq\Lambda(s)$ is a $\preceq$-saturated set, then, by \cite[Prop~3.1(ii)]{LT}, the type $D_s$ linkage principle holds in $\mc C_\Lambda$. So if $\lambda,\mu\in\Lambda$ belong to the same $\mc C_\Lambda$-block, then they are conjugate under the dot action of $W_p(D_s)$.
For $\lambda\in\Lambda(s)$ we define the projection functor $\widetilde{\rm pr}_\lambda:\mc C_{\Lambda(s)}\to\mc C_{\Lambda(s)}$ by $\widetilde{\rm pr}_\lambda M=O_{W_p(D_s)\cdot\lambda\cap X^+}(M)$. Then $M=\bigoplus_\lambda\widetilde{\rm pr}_\lambda M$ where the sum is over a set of representatives of the type $D_s$ linkage classes in $\Lambda(s)$. Note that $\widetilde{\rm pr}_\lambda M$ is a direct summand of ${\rm pr}_\lambda M$.
Now let $\lambda,\lambda'\in\Lambda(s)$ with $\lambda'\in{\rm Supp}(\lambda)$ and let $\mc C,\mc C'$ be Serre subcategories of $\mc C_{\Lambda(s)}$ such that ${\rm pr}_{\lambda'}((\widetilde{\rm pr}_\lambda M)\ot V)\in\mc C_{\Lambda(s)}$ for all $M\in\mc C$ and ${\rm pr}_\lambda((\widetilde{\rm pr}_{\lambda'}M)\ot V)\in\mc C_{\Lambda(s)}$ for all $M\in\mc C'$. Then we define the \emph{translation functors} $\widetilde T_\lambda^{\lambda'}:\mc C\to\mc C_{\Lambda(s)}$ and $\widetilde T_{\lambda'}^\lambda:\mc C'\to\mc C_{\Lambda(s)}$ by $\widetilde T_\lambda^{\lambda'}M=\widetilde{\rm pr}_{\lambda'}((\widetilde{\rm pr}_\lambda M)\ot V)$ and $\widetilde T_{\lambda'}^\lambda M=\widetilde{\rm pr}_\lambda((\widetilde{\rm pr}_{\lambda'} M)\ot V)$. Note that if $\mu\in X^+\cap  W_p(D_s)\cdot\lambda$ and $\nabla(\mu)\in\mc C$, then $\widetilde T_\lambda^{\lambda'}\nabla(\mu)$ has a good filtration with sections $\nabla(\nu)$, $\nu\in{\rm Supp}(\mu)\cap  W_p(D_s)\cdot\lambda'$. The analogue for Weyl modules and Weyl filtrations also holds.
If $\widetilde T_\lambda^{\lambda'}$ and $\widetilde T_{\lambda'}^\lambda$ have image in $\mc C$ and $\mc C'$, then they restrict to functors $\mc C\to\mc C'$ and $\mc C'\to \mc C$ which are exact and each others left and right adjoint.
To unify notation with the case $G=\GL_n$ we put ${\rm Supp}_h={\rm Supp}$ for $h\in\{1,2\}$.

We now return to the general case $G=\GL_n$ or $G=\Sp_n$. Proposition~\ref{prop.trans_equivalence} below is a combination of Propositions~4.1 in \cite{LT} and \cite{T} and Proposition~\ref{prop.trans_projection} below is a combination of Propositions~4.2 in \cite{LT} and \cite{T}. In the case $G=\Sp_n$ we can ignore the subscripts $h$ and $\ov h$, and in the $G=\GL_n$-case we can read $\widetilde T$ as $T$.

\begin{propgl}[Translation equivalence]\label{prop.trans_equivalence}
Let $h,\ov h\in\{1,2\}$ be distinct, let $\lambda,\lambda'\in\Lambda(s)$ with $\lambda'\in{\rm Supp}_h(\lambda)$ and let $\Lambda\subseteq H\cdot\lambda\cap\Lambda(s),\Lambda'\subseteq H\cdot\lambda'\cap\Lambda(s)$ be $\preceq$-saturated sets. 
Assume
\begin{enumerate}[{\rm (1)}]
\item ${\rm Supp}_h(\nu)\cap W_p\cdot\lambda'\subseteq\Lambda(s)$ for all $\nu\in\Lambda$, and ${\rm Supp}_{\ov h}(\nu')\cap  W_p\cdot\lambda\subseteq\Lambda(s)$ for all $\nu'\in\Lambda'$.
\item $|{\rm Supp}_h(\nu)\cap H\cdot\lambda'|=1=|{\rm Supp}_{\ov h}(\nu')\cap H\cdot\lambda|$ for all $\nu\in\Lambda$ and $\nu'\in\Lambda'$.
\item The map $\nu\mapsto\nu':\Lambda\to\Lambda(s)$ given by ${\rm Supp}_h(\nu)\cap H\cdot\lambda'=\{\nu'\}$ has image $\Lambda'$, and together with its inverse $\Lambda'\to\Lambda$ it preserves the order $\preceq$.
\end{enumerate}
Then $\widetilde T_\lambda^{\lambda'}$ restricts to an equivalence of categories $\mc C_\Lambda\to\mc C_{\Lambda'}$ with inverse $\widetilde T_{\lambda'}^\lambda:\mc C_{\Lambda'}\to\mc C_\Lambda$.
Furthermore, with $\nu$ and $\nu'$ as in (3), we have $\widetilde T_\lambda^{\lambda'}\nabla(\nu)=\nabla(\nu')$, $\widetilde T_\lambda^{\lambda'}\Delta(\nu)=\Delta(\nu')$, $\widetilde T_\lambda^{\lambda'}L(\nu)=L(\nu')$, $\widetilde T_\lambda^{\lambda'}T(\nu)=T(\nu')$ and $\widetilde T_\lambda^{\lambda'}I_\Lambda(\nu)=I_{\Lambda'}(\nu')$.
\end{propgl}

\begin{propgl}[Translation projection]\label{prop.trans_projection}
Let $h,\ov h\in\{1,2\}$ be distinct, let $\lambda,\lambda'\in\Lambda(s)$ with $\lambda'\in{\rm Supp}_h(\lambda)$ and let $\Lambda\subseteq H\cdot\lambda\cap\Lambda(s),\Lambda'\subseteq H\cdot\lambda'\cap\Lambda(s)$ be $\preceq$-saturated sets. Put $\tilde\Lambda=\{\nu\in\Lambda\,|\,{\rm Supp}_h(\nu)\cap H\cdot\lambda'\ne\emptyset\}$. Assume
\begin{enumerate}[{\rm (1)}]
\item ${\rm Supp}_h(\nu)\cap W_p\cdot\lambda'\subseteq\Lambda(s)$ for all $\nu\in\Lambda$, and ${\rm Supp}_{\ov h}(\nu')\cap W_p\cdot\lambda\subseteq\Lambda(s)$ for all $\nu'\in\Lambda'$.
\item $|{\rm Supp}_h(\nu)\cap H\cdot\lambda'|=1$ for all $\nu\in\tilde\Lambda$, and $|{\rm Supp}_{\ov h}(\nu')\cap H\cdot\lambda|=2$ for all $\nu'\in\Lambda'$.
\item The map $\nu\mapsto\nu':\tilde\Lambda\to\Lambda(s)$ given by ${\rm Supp}_h(\nu)\cap H\cdot\lambda'=\{\nu'\}$ is a 2-to-1 map which has image $\Lambda'$ and preserves the order $\preceq$. For $\nu'\in\Lambda'$ we can write ${\rm Supp}_{\ov h}(\nu')\cap H\cdot\lambda=\{\nu^+,\nu^-\}$ with $\nu^-\prec\nu^+$ and then we have $\Hom_G(\nabla(\nu^+),\nabla(\nu^-))\ne0$ and $\eta'\preceq\nu'\Rightarrow\eta^+\preceq\nu^+$ and $\eta^-\preceq\nu^-$.
\end{enumerate}
Then $\widetilde T_\lambda^{\lambda'}$ restricts to a functor $\mc C_\Lambda\to\mc C_{\Lambda'}$ and $\widetilde T_{\lambda'}^\lambda$ restricts to a functor $\mc C_{\Lambda'}\to\mc C_\Lambda$.
Now let $\nu\in\Lambda$. If $\nu\notin\tilde\Lambda$, then $\widetilde T_\lambda^{\lambda'}\nabla(\nu)=\widetilde T_\lambda^{\lambda'}\Delta(\nu)=\widetilde T_\lambda^{\lambda'}L(\nu)=0$.
For $\nu'\in\Lambda'$ with $\nu^\pm$ as in (3), we have $\widetilde T_\lambda^{\lambda'}\nabla(\nu^\pm)=\nabla(\nu')$, $\widetilde T_\lambda^{\lambda'}\Delta(\nu^\pm)=\Delta(\nu')$, $\widetilde T_\lambda^{\lambda'}L(\nu^-)=L(\nu')$, $\widetilde T_\lambda^{\lambda'}L(\nu^+)=0$, $\widetilde T_{\lambda'}^\lambda T(\nu')=T(\nu^+)$ and $\widetilde T_{\lambda'}^\lambda I_{\Lambda'}(\nu')=I_\Lambda(\nu^-)$.
\end{propgl}

\begin{remsgl}\label{rems.translation}
1.\ It is easy to see that in the situation of Proposition~\ref{prop.trans_projection} we have a nonsplit extension $$0\to\nabla(\nu^-)\to\widetilde T_{\lambda'}^\lambda\nabla(\nu')\to\nabla(\nu^+)\to0\,:$$ if it were split, then $\dim\Hom_G(\nabla(\nu^+),\widetilde T_{\lambda'}^\lambda\nabla(\nu'))>1$, but using the adjoint functor property it is clear that this dimension is $1$. If we now consider the long exact cohomology sequence associated to the above short exact sequence and the functor $\Hom_G(\nabla(\nu^+),-)$, and we also use the adjoint functor property (which holds for all $\Ext_G^i$), then we obtain $\dim\Hom_G(\nabla(\nu^+),\nabla(\nu^-))=\dim\Ext^1_G(\nabla(\nu^+),\nabla(\nu^-))=1$ and $\Ext^i_G(\nabla(\nu^+),\nabla(\nu^-))=0$ for $i>1$. See also \cite[II.2.14 and 4.13]{Jan}.\\
2.\ From the proofs of Theorems~6.1 in \cite{LT} and \cite{T} we deduce that the assumptions of Proposition~\ref{prop.trans_equivalence} are satisfied in the following situations
where we will always take $\Lambda=\Lambda(s)\cap H\cdot\lambda$ and $\Lambda'=\Lambda(s)\cap H\cdot\lambda'$ once we have chosen $\lambda,\lambda'\in\Lambda(s)$. We will derive the ``moves" from $co_\lambda$ rather than from $c_\lambda$ as in \cite[Thm~6.1]{LT} and \cite[Thm~6.1]{T}. If $co_\lambda$ is of the form
$\xy
(0,0)="a1";(0,0)*{\cdots};
(5,0)="a2";(5,-1)*{\land};
(10,0)="a3";(10,0)*{\cdots};
(15,0)="a4";(15,0)*{\bullet};
(20,0)="a5";(20,1)*{\vee};(20,-4)*{a};
(25,0)="a6";(25,0)*{\cdots};(25,-6.5)*{\ }; 
"a5";"a2"**\crv{(20,8)&(5,6)}; 
\endxy\ $ when $G=\GL_n$ or $G=\Sp_n$ and the cap is to the right of the wall, then we choose
$co_\lambda'=
\xy
(0,0)="a1";(0,0)*{\cdots};
(5,0)="a2";(5,-1)*{\land};
(10,0)="a3";(10,0)*{\cdots};
(15,0)="a4";(15,1)*{\vee};
(20,0)="a5";(20,0)*{\bullet};(20,-3)*{a};
(25,0)="a6";(25,0)*{\cdots};(25,-6.5)*{\ }; 
"a4";"a2"**\crv{(15,7)&(5,6)}; 
\endxy\ .$
If $G=\Sp_n$ and the cap is to the left of the wall, we let the curves go below the horizontal line.
The bijection $\nu\mapsto\nu':\Lambda\to\Lambda'$ is then given by
\parbox[c][.9cm]{3.2cm}{$\begin{smallmatrix}
\cdots&{\rm o}&\land&\cdots&\mapsto&\cdots&\land&{\rm o}&\cdots\\
\cdots&{\rm o}&\vee&\cdots&\mapsto&\cdots&\vee&{\rm o}&\cdots\\
&&a&&&&&a&
\end{smallmatrix}
$}.
In case the $(a-1)$-node in the arrow diagram of $\lambda$ carries an $\times$ we get an $\times$ at the $a$-node of $\lambda'$ and the bijection $\nu\mapsto\nu':\Lambda\to\Lambda'$
is then given by
\parbox[c][.9cm]{3.2cm}{$\begin{smallmatrix}
\cdots&\times&\land&\cdots&\mapsto&\cdots&\land&\times&\cdots\\
\cdots&\times&\vee&\cdots&\mapsto&\cdots&\vee&\times&\cdots\\
&&a&&&&&a&
\end{smallmatrix}
$}.
One can also move the right end node of the cap one step to the right past an empty node or past an $\times$ (that is just the inverse bijection) or move the left end node one step past an empty node or past an $\times$. In case $G=\Sp_n$ one can also move one of the end nodes of a curl one step past an empty node or past an $\times$. The diagrammatic descriptions of the bijections are the same.
Furthermore, one can turn a curl with left end node at 0 into a cap by replacing the arrow at the 0-node by its opposite. In terms of the weights this bijection is just the identity. Finally, one can turn a curl with right end node the last node into a cap by replacing the arrow at the last node by its opposite. The bijection $\nu\mapsto\nu':\Lambda\to\Lambda'$ is given by 
\parbox[c][.9cm]{2.3cm}{$\begin{smallmatrix}
\cdots&\vee&\mapsto&\cdots&\land\\
\cdots&\land&\mapsto&\cdots&\vee
\end{smallmatrix}
$.}
In most applications we start with a cap or curl with no caps or curls inside it and then repeatedly apply moves as above until we have a cap with consecutive end nodes. Then we can apply the next remark.\\
3.\ From the aforementioned proofs we can also deduce that the assumptions of Proposition~\ref{prop.trans_projection} are satisfied in the following situations where we will always take $\Lambda=\Lambda(s)\cap H\cdot\lambda$ and $\Lambda'=\Lambda(s)\cap H\cdot\lambda'$ once we have chosen $\lambda,\lambda'\in\Lambda(s)$. We will derive the ``moves" from $co_\lambda$ rather than $c_\lambda$ as in \cite[Thm~6.1]{LT} and \cite[Thm~6.1]{T}, so we will have $\lambda=\lambda^-$, rather than $\lambda=\lambda^+$.
The set $\tilde\Lambda$ will always consist of the $\nu\in\Lambda$ for which the cap or curl of $co_\lambda$ under consideration is oriented in $co_{\lambda\nu}$.
If $co_\lambda$ is of the form
$\xy
(0,0)="a1";(0,0)*{\cdots};
(5,0)="a2";(5,-1)*{\land};
(10,0)="a3";(10,1)*{\vee};(10,-4)*{a};
(15,0)="a4";(15,0)*{\cdots};
"a3";"a2"**\crv{(10,5)&(5,5)}; 
\endxy\ $ when $G=\GL_n$ or $G=\Sp_n$ and the cap is to the right of the wall, or $co_\lambda=
\xy
(0,0)="a1";(0,0)*{\cdots};
(5,0)="a2";(5,-1)*{\land};
(10,0)="a3";(10,1)*{\vee};(10,-4)*{\,a};
(15,0)="a4";(15,0)*{\cdots};(15,-7)*{\ }; 
"a3";"a2"**\crv{(10,-5)&(5,-5)}; 
\endxy\ $ when $G=\Sp_n$ and the cap is to the left of the wall, then we choose
$\lambda'=
\xy
(0,0)="a1";(0,0)*{\cdots};
(5,0)="a2";(5,0)*{\bullet};
(10,0)="a3";(10,1)*{\vee};(10,-1)*{\land};(10,-4)*{a};
(15,0)="a4";(15,0)*{\cdots};(15,-7)*{\ }; 
\endxy\ .$
and the projection $\nu\mapsto\nu':\tilde\Lambda\to\Lambda'$ is given by
\parbox[c][.9cm]{3.55cm}{$\begin{smallmatrix}
\cdots&\land&\vee&\cdots\\
\cdots&\vee&\land&\cdots
\end{smallmatrix}
\mapsto
\begin{smallmatrix}
\cdots&{\rm o}&\times&\cdots&\\
\end{smallmatrix}
$}.
In case $G=\GL_n$ or $G=\Sp_n$ and $a>1$ we can also choose
$\lambda'=
\xy
(0,0)="a1";(0,0)*{\cdots};
(5,0)="a3";(5,1)*{\vee};(5,-1)*{\land};
(10,0)="a3";(10,0)*{\bullet};(10,-4)*{a};
(15,0)="a4";(15,0)*{\cdots};(15,-7)*{\ }; 
\endxy\ $
and then the projection $\nu\mapsto\nu':\tilde\Lambda\to\Lambda'$ is given by
\parbox[c][.9cm]{3.55cm}{$\begin{smallmatrix}
\cdots&\land&\vee&\cdots\\
\cdots&\vee&\land&\cdots
\end{smallmatrix}
\mapsto
\begin{smallmatrix}
\cdots&\times&{\rm o}&\cdots
\end{smallmatrix}
$}.
Finally, if $co_\lambda$ has a curl at the first two nodes, then we can choose
$\lambda'=
\xy
(0,0)="a1";(0,0)*{\bullet};
(5,0)="a2";(5,1)*{\vee};(5,-1)*{\land};(5,-4)*{a};
(10,0)="a3";(10,0)*{\cdots};(10,-7)*{\ }; 
\endxy\ .$
and the projection $\nu\mapsto\nu':\tilde\Lambda\to\Lambda'$ is given by
\parbox[c][.9cm]{2.7cm}{$\begin{smallmatrix}
\land&\land&\cdots\\
\vee&\vee&\cdots
\end{smallmatrix}
\mapsto
\begin{smallmatrix}
{\rm o}&\times&\cdots
\end{smallmatrix}
$}. However, this is just the combination of the trivial move mentioned near the end of the previous remark and the above ``cap-contraction".
\end{remsgl}

\section{The polynomials}\label{s.polynomials}
Recall the definitions of $\preceq$, $\Lambda(s)$ and $H$ from Section~\ref{ss.reductive_groups}. 
Throughout this section we assume that $\Lambda,\Lambda'\subseteq\Lambda(s)$ are the intersection of $\Lambda(s)$ with an $H$-orbit under the dot action.
\begin{defgl}
For $\lambda,\mu\in\Lambda$ we define the polynomials $d_{\lambda\mu}\in\mb Z[q]$ by $$d_{\lambda\mu}=\begin{cases}
q^{\text{number of clockwise caps and curls in\ } c_{\lambda\mu}}, &\text{if $\mu\preceq\lambda$ and $c_{\lambda\mu}$ is oriented,}\\
0&\text{otherwise.}
\end{cases}$$
Clearly the matrix $(d_{\lambda\mu})_{\lambda,\mu}$ is lower uni-triangular for the ordering $\preceq$:
$d_{\lambda\lambda}=1$ and $d_{\lambda\mu}\ne0$ implies $\mu\preceq\lambda$. Next we define the polynomials $p_{\lambda\mu}$
by requiring that the matrix $(p_{\lambda\mu})_{\lambda,\mu}$ is the inverse of $(d_{\lambda\mu}(-q))_{\lambda,\mu}$.
This inverse is then also lower uni-triangular for the ordering $\preceq$.
\end{defgl}

\begin{defgl}
For $\lambda,\mu\in\Lambda$ we define the polynomials $e_{\lambda\mu}\in\mb Z[q]$ by $$e_{\lambda\mu}=\begin{cases}
q^{\text{number of anti-clockwise caps and curls in\ } co_{\mu\lambda}}, &\text{if $\mu\preceq\lambda$ and $co_{\mu\lambda}$ is oriented,}\\
0&\text{otherwise.}
\end{cases}$$
Clearly the matrix $(e_{\lambda\mu})_{\lambda,\mu}$ is lower uni-triangular for the ordering $\preceq$. Next we define the polynomials $r_{\lambda\mu}$
by requiring that the matrix $(r_{\lambda\mu})_{\lambda,\mu}$ is the inverse of $(e_{\lambda\mu}(-q))_{\lambda,\mu}$.
This inverse is then also lower uni-triangular for the ordering $\preceq$.
\end{defgl}

\begin{remgl}\label{rem.previous_results}
Let $\lambda,\mu\in\Lambda_p$. In \cite[Thm~6.1]{LT} and \cite[Thm~6.1]{T} it was shown that $(T(\lambda):\nabla(\mu))=d_{\lambda\mu}(1)$ if $\mu\preceq\lambda$, and 0 otherwise, and in \cite[Thm~7.1]{LT} and \cite[Thm~7.1]{T} it was shown that $[\nabla(\lambda):L(\mu)]=e_{\lambda\mu}(1)$ if $\mu\preceq\lambda$, and 0 otherwise.
\end{remgl}

The proof of the following lemma is easy, we leave it to the reader.
\begin{lemgl}\label{lem.d_and_e_eqn}\ 
\begin{enumerate}[{\rm (i)}]
\item Assume we are in the situation of Remark~\ref{rems.translation}.2. Then $d_{\lambda\mu}=d_{\lambda'\mu'}$ and $e_{\lambda\mu}=e_{\lambda'\mu'}$ for $\lambda,\mu\in\Lambda$. 
\item Assume we are in the situation of Remark~\ref{rems.translation}.3. 
For $\lambda=\lambda^+\in\Lambda$ we have
\begin{align*}
d_{\lambda^+\mu^+}&=d_{\lambda'\mu'},\text{\qquad for\ }\mu^+\in\Lambda,\\
d_{\lambda^+\mu^-}&=qd_{\lambda'\mu'},\text{\qquad for\ }\mu^-\in\Lambda,\\
d_{\lambda^+\mu}&=0\text{\qquad for\ }\mu\in\Lambda\text{\ not of the form\ }\mu^\pm.
\end{align*}
For $\mu=\mu^-\in\Lambda$ we have
\begin{align*}
e_{\lambda^-\mu^-}&=e_{\lambda'\mu'},\text{\qquad for\ }\lambda^-\in\Lambda,\\
e_{\lambda^+\mu^-}&=qe_{\lambda'\mu'},\text{\qquad for\ }\lambda^+\in\Lambda\\
e_{\lambda\mu^-}&=0\text{\qquad for\ }\lambda\in\Lambda\text{\ not of the form\ }\lambda^\pm.
\end{align*}
\end{enumerate}
\end{lemgl}
As in \cite[Sect.~7]{CdV} one can show that
\begin{align*}
d_{\lambda^+\mu^+}&=q^{-1}d_{\lambda^-\mu^+}+d_{\lambda^-\mu^-}\text{\ \ and}\\
d_{\lambda^+\mu^-}&=qd_{\lambda^-\mu^-}+d_{\lambda^-\mu^+}\,.
\end{align*}


\begin{propgl}\label{prop.p_and_r_eqn}\ 
\begin{enumerate}[{\rm (i)}]
\item Assume we are in the situation of Remark~\ref{rems.translation}.2. Then $p_{\lambda\mu}=p_{\lambda'\mu'}$ and $r_{\lambda\mu}=r_{\lambda'\mu'}$ for $\lambda,\mu\in\Lambda$. 
\item Assume we are in the situation of Remark~\ref{rems.translation}.3.  
For $\lambda=\lambda^+\in\Lambda$ we have
\begin{align}\label{eq.p_trans_proj1} 
p_{\lambda^+\mu^+}&=p_{\lambda'\mu'}+qp_{\lambda^-\mu^+},
\end{align}
for all $\mu^+\in\Lambda$ and
\begin{align}\label{eq.p_trans_proj2}
p_{\lambda^+\mu}&=qp_{\lambda^-\mu}
\end{align}
for all $\mu\in\Lambda$ not of the form $\mu^+$.\\
For $\mu=\mu^-\in\Lambda$ we have
\begin{align}\label{eq.r_trans_proj1} 
r_{\lambda^-\mu^-}&=r_{\lambda'\mu'}+qr_{\lambda^-\mu^+},
\end{align}
for all $\lambda^-\in\Lambda$ and
\begin{align}\label{eq.r_trans_proj2}
r_{\lambda\mu^-}&=qr_{\lambda\mu^+}
\end{align}
for all $\lambda\in\Lambda$ not of the form $\lambda^-$.
\end{enumerate}
\end{propgl}
\begin{proof}
(i).\ By Lemma~\ref{lem.d_and_e_eqn}(i) the matrices $(d_{\lambda\mu})_{\lambda,\mu}$ and $(d_{\lambda'\mu'})_{\lambda,\mu}$ are the same, so
their inverses $(p_{\lambda\mu})_{\lambda,\mu}$ and $(p_{\lambda'\mu'})_{\lambda,\mu}$ are also the same. The second identity is proved in the same way.\\
(ii).\ We will prove \eqref{eq.p_trans_proj1} and \eqref{eq.p_trans_proj2} by $\succeq$-induction on $\mu$ with $\mu=\lambda$ as (trivial) basis case.
Put $\tilde p_{\lambda\mu}=p_{\lambda\mu}(-q)$. By the definition of the $p_{\lambda\mu}$ and the induction hypothesis we have
\begin{align*}
\tilde p_{\lambda\mu}&=\delta_{\lambda\mu}-\sum_{\mu\prec\nu\preceq\lambda}\tilde p_{\lambda\nu}d_{\nu\mu}\\
&=\delta_{\lambda\mu}-\mathop{\sum_{\mu\prec\nu\preceq\lambda}}_{\nu=\nu^+}(\tilde p_{\lambda'\nu'}-q\tilde p_{\lambda^-\nu})d_{\nu\mu}+
q\mathop{\sum_{\mu\prec\nu\preceq\lambda}}_{\nu\ne\nu^+}\tilde p_{\lambda^-\nu}d_{\nu\mu}\\
&=\bigg(\delta_{\lambda\mu}-\mathop{\sum_{\mu\prec\nu\preceq\lambda}}_{\nu=\nu^+}\tilde p_{\lambda'\nu'}d_{\nu\mu}\bigg)+\bigg(q\sum_{\mu\prec\nu\preceq\lambda^-}\tilde p_{\lambda^-\nu}d_{\nu\mu}\bigg)\,.
\end{align*}
The second bracketed expression equals $q\delta_{\lambda^-\mu}-q\tilde p_{\lambda^-\mu}$ by the definition of the $p_{\lambda\mu}$. Denote the first bracketed expression by $E$.
Then we have $$E=\begin{cases}
\delta_{\lambda'\mu'}-\sum_{\mu'\prec\nu'\preceq\lambda'}\tilde p_{\lambda'\nu'}d_{\nu'\mu'}=\tilde p_{\lambda'\mu'}&\text{if\ }\mu=\mu^+,\\
-q\sum_{\mu'\preceq\nu'\preceq\lambda'}\tilde p_{\lambda'\nu'}d_{\nu'\mu'}=-q\delta_{\lambda'\mu'}=-q\delta_{\lambda^-\mu}&\text{if\ }\mu=\mu^-,\\ 
0&\text{if\ }\mu\ne\mu^\pm,
\end{cases}$$
where we used Lemma~\ref{lem.d_and_e_eqn}(ii) and that, when $\mu=\mu^-$, we can have $\nu=\mu^+$ in the sum in $E$.
It follows that $$\tilde p_{\lambda\mu}=\begin{cases}
\tilde p_{\lambda'\mu'}-q\tilde p_{\lambda^-\mu}&\text{if\ }\mu=\mu^+,\\
-q\tilde p_{\lambda^-\mu}&\text{if\ }\mu\ne\mu^+,
\end{cases}$$ as required.

Equations \eqref{eq.r_trans_proj1} and \eqref{eq.r_trans_proj2} are proved by $\preceq$-induction on $\lambda$ with $\lambda=\mu$ as (trivial) basis case.
Put $\tilde r_{\lambda\mu}=r_{\lambda\mu}(-q)$. By the definition of the $r_{\lambda\mu}$ and the induction hypothesis we have
\begin{align*}
\tilde r_{\lambda\mu}&=\delta_{\lambda\mu}-\sum_{\mu\preceq\nu\prec\lambda}e_{\lambda\nu}\tilde r_{\nu\mu}\\
&=\delta_{\lambda\mu}-\mathop{\sum_{\mu\preceq\nu\prec\lambda}}_{\nu=\nu^-}e_{\lambda\nu}(\tilde r_{\nu'\mu'}-q\tilde r_{\nu\mu^+})+
q\mathop{\sum_{\mu\preceq\nu\prec\lambda}}_{\nu\ne\nu^-}e_{\lambda\nu}\tilde r_{\nu\mu^+}\\
&=\bigg(\delta_{\lambda\mu}-\mathop{\sum_{\mu\preceq\nu\prec\lambda}}_{\nu=\nu^-}e_{\lambda\nu}\tilde r_{\nu'\mu'}\bigg)+\bigg(q\sum_{\mu^+\preceq\nu\prec\lambda}e_{\lambda\nu}\tilde r_{\nu\mu^+}\bigg)\,.
\end{align*}
We leave the rest of the proof to the reader.
\end{proof}

\begin{remsgl}\label{rems.polynomials}
1.\ Obviously $\lambda\ne\mu\Rightarrow d_{\lambda\mu},e_{\lambda\mu}\in q\mb Z[q]$, so we also have $\lambda\ne\mu\Rightarrow p_{\lambda\mu},r_{\lambda\mu}\in q\mb Z[q]$.\\
2.\ Using elementary properties of $n(\lambda,\mu)$ and the $l_i(\eta,\xi)$, see e.g. \cite[p175,176]{CdV}, one can easily show by induction that $p_{\lambda\mu}\ne0\Leftrightarrow\mu\preceq\lambda$
and that $p_{\lambda\mu}\ne0\Rightarrow\deg(p_{\lambda\mu})=n(\lambda,\mu)$ and the degrees of the terms of $p_{\lambda\mu}$ have the same parity.
Similarly, we obtain $r_{\lambda\mu}\ne0\Leftrightarrow\mu\preceq\lambda$ and $r_{\lambda\mu}\ne0\Rightarrow\deg(r_{\lambda\mu})=n(\lambda,\mu)$ and the degrees of the terms of $r_{\lambda\mu}$ have the same parity.\\
3.\ If $G=\Sp_n$ and the arrow diagram of any $\lambda\in\Lambda$ has an arrow at $0$, then assume that the parity of the number of $\land$'s in the arrow diagrams of the weights in $\Lambda$ is fixed.
Let $\lambda,\mu\in\Lambda$ and let $(\eta^1,\eta^2)$ and $(\xi^1,\xi^2)$ be the associated pairs of $(\land,\vee)$-sequences. Then it is easy to see that the polynomials $d_{\lambda\mu}$, $p_{\lambda\mu}$, $e_{\lambda\mu}$ and $r_{\lambda\mu}$
only depend on $(\eta^1,\eta^2)$ and $(\xi^1,\xi^2)$. In fact one can define the cap and cap-curl diagrams for $(\land,\vee)$-sequences: Just do this as on the left side of the wall in the $\Sp_n$-case and as on any side of the wall in the $\GL_n$-case.
This is essentially the same as in \cite[Sect~4.5]{CdV}: In the diagram from \cite[Sect~8]{CdV} in the $\Sp_n$-Brauer-case we have to put in the wall using their $\rho_\delta$ rather than our $\rho$ and omit the infinite tail $\land\cdots\land\vee\vee\cdots$ starting at the wall. The associated $(\land,\vee)$-sequence is then formed by the remaining single arrows to the left of the wall. In the $\GL_n$-walled Brauer-case we have to put in the walls using their $\rho_\delta$ rather than our $\rho$ ($(\rho_\delta)_i=\delta-i+1$ for $i\ge1$) and omit the infinite tail of $\vee$'s to the right of the wall above the line and the infinite tail of $\land$'s to the left of the wall below the line. The associated $(\land,\vee)$-sequence is then formed by the remaining single arrows between the walls.
We omit the infinite rays in the cap(-curl) diagram from \cite[Sect~8]{CdV}.
Then we can also define the $d$ and $p$ polynomials for $(\land,\vee)$-sequences\footnote{Of course, the meaning of $c_{\xi\eta}$, $d_{\xi\eta}$ and $p_{\xi\eta}$ depends on the case: $G=\GL_n$ or $G=\Sp_n$.} and we then have $$d_{\lambda\mu}=d_{\eta^1\xi^1}d_{\eta^2\xi^2}\,.$$
So the matrix $(d_{\lambda\mu})_{\lambda,\mu}$ is the Kronecker product of the matrices $(d_{\eta^1\xi^1})_{\eta^1,\xi^1}$ and $(d_{\eta^2\xi^2})_{\eta^2,\xi^2}$, where the $\eta^i$ and the $\xi^i$ vary over a strong conjugacy class when $G=\GL_n$ and over a conjugacy class when $G=\Sp_n$. But then the same must hold for their inverses, so we obtain: $$p_{\lambda\mu}=p_{\eta^1\xi^1}p_{\eta^2\xi^2}\,.$$
The analogues of Remark~\ref{rems.translation}.3, Lemma~\ref{lem.d_and_e_eqn}(ii) and Proposition~\ref{prop.p_and_r_eqn}(ii) for $d$ and $p$-polynomials associated to $(\land,\vee)$-sequences also hold.
Next we could define codiagrams and $e$ and $r$-polynomials for $(\land,\vee)$-sequences, but instead we use the order reversing involution $\dagger$ which replaces every arrow by its opposite, and then we have
$$e_{\lambda\mu}=d_{(\xi^1)^\dagger(\eta^1)^\dagger}d_{(\xi^2)^\dagger(\eta^2)^\dagger}\text{\ and\ }r_{\lambda\mu}=p_{(\xi^1)^\dagger(\eta^1)^\dagger}p_{(\xi^2)^\dagger(\eta^2)^\dagger}\,.$$
One can also define $\dagger$ on $\Lambda(s)$ and then obtain the identities
$$e_{\lambda\mu}=d_{\mu^\dagger\lambda^\dagger}\text{\ and\ }r_{\lambda\mu}=p_{\mu^\dagger\lambda^\dagger}\,,$$
but in the case of $G=\GL_n$ it is only clear that this works when $s_1=s_2$, since otherwise the values of $s_1$ and $s_2$ swap and the walls would move.
See also \cite[Cor to Thm~7.1]{T} and \cite[Cor to Thm~7.1]{LT}.

Finally, we point out that we have an explicit combinatorial formula for the $p_{\eta\xi}$ as in \cite[Sect~8]{CdV} (in the $\GL_n$-case see also \cite[Sect~5]{BS}).
In both cases we work with a single external/unbounded chamber (and omit all the infinite rays).
Of course in \cite[Sect~8]{CdV} (and \cite[Sect~5]{BS}) this expression is actually the definition of their $p$-polynomials, but one can prove as in \cite[Sect~8]{CdV} 
that this alternative definition leads to the same recursive relations as \eqref{eq.p_trans_proj1} and \eqref{eq.p_trans_proj2} for $(\land,\vee)$-sequences.
\end{remsgl}

\section{Tilting and injective resolutions}\label{s.resolutions}
We retain the notation and assumptions from the previous section. For $\lambda,\mu\in\Lambda$ define the integers $p^i_{\lambda\mu}, r^i_{\lambda\mu}\in\mb Z$ by
\begin{align*}
p_{\lambda\mu}(q)&=\sum_{i\ge0}p^i_{\lambda\mu}q^i\text{\ and}\\
r_{\lambda\mu}(q)&=\sum_{i\ge0}r^i_{\lambda\mu}q^i\,.
\end{align*}

The theorem below is the analogue of \cite[Thm~5.3]{BS} and \cite[Thm~9.1]{CdV} in our setting.
\begin{thmgl}\label{thm.tilting_resolution}
The induced module $\nabla(\lambda)$, $\lambda\in\Lambda$, has a finite left tilting resolution:
$$\cdots\to T^1(\lambda)\to T^0(\lambda)\to\nabla(\lambda)\to 0$$
where
$$T^i(\lambda)=\bigoplus_{\mu\in\Lambda} p^i_{\lambda\mu}T(\mu)\,.$$
\end{thmgl}
\begin{proof}
The proof is very similar to that of \cite[Thm~9.1]{CdV}. One merely has to replace $\Delta_{(a)}(\mu)$, $P_{(a)}(\mu)$, $P^i_{(a)}(\mu)$ and ${\rm res}^\lambda_{(a+1)}$ in that proof by $\nabla(\mu)$, $T(\mu)$, $T^i(\mu)$ and $\widetilde T_{\lambda'}^\lambda $, and for the extension to a chain map
use the fact that $\Hom_G(T^i(\mu),-)$ maps short exact sequences of modules with a good filtration to exact sequences. 
We leave the details to the reader and give the proof of the next theorem in more detail.
\end{proof}

\begin{thmgl}\label{thm.injective_resolution}
The induced module $\nabla(\mu)$, $\mu\in\Lambda$, has a finite injective resolution in $\mc C_\Lambda$:
$$0\to \nabla(\mu)\to I^0(\mu)\to I^1(\mu)\to\cdots$$
where
$$I^i(\mu)=\bigoplus_{\lambda\in\Lambda} r^i_{\lambda\mu}I_\Lambda(\lambda)\,.$$
\end{thmgl}
\begin{proof}
The proof follows \cite[Thm~9.1]{CdV} and \cite[Thms~6.1,7.1]{LT} (and \cite[Thms~6.1,7.1]{T}). 
We assume that the assertion holds for weights $\nu\in\Lambda$ with $\nu\succ\mu$ and weights whose cap-curl codiagram has fewer caps and curls in case $G=\Sp_n$ and whose cap codiagram has fewer caps in case $G=\GL_n$. If $co_\mu$ has no caps or curls, then $\mu$ is maximal in $\Lambda$ and $\nabla(\mu)=I_\Lambda(\mu)$. Now assume $co_\mu$ has a cap or curl.
After finitely many translation equivalences, see Remark~\ref{rems.translation}.2, we may assume that there is a cap in $co_\mu$ connecting consecutive vertices.
Fix such a cap. Then we are in the situation of Remark~\ref{rems.translation}.3 and we can write $\mu=\mu^-$. By the inductive assumption we have injective resolutions
\begin{align}\label{eq.res_mu'}
0\to \nabla(\mu')\to I^0(\mu')\to I^1(\mu')\to\cdots
\end{align}
in $\mc C_{\Lambda'}$ and
\begin{align}\label{eq.res_mu^+}
0\to \nabla(\mu^+)\to I^0(\mu^+)\to I^1(\mu^+)\to\cdots
\end{align}
in $\mc C_\Lambda$.
Recall from Remark~\ref{rems.translation}.1 in Section~\ref{ss.translation} that we have an exact sequence
\begin{align}\label{eq.extension}
0\to\nabla(\mu^-)\to\widetilde T_{\mu'}^\mu\nabla(\mu')\stackrel{f}{\to}\nabla(\mu^+)\to0
\end{align}

Applying $\widetilde T_{\mu'}^\mu$ to \eqref{eq.res_mu'} and extending $f$ to a chain map using \eqref{eq.res_mu^+} we obtain a commutative diagram with exact rows
$$
\hspace{-.2cm}
\xymatrix @R=15pt @C=15pt @M=6pt{
0\ar[r]&\widetilde T_{\mu'}^\mu\nabla(\mu')\ar[r]\ar[d]^f&
\widetilde T_{\mu'}^\mu I^0(\mu')\ar[r]\ar[d]&
\widetilde T_{\mu'}^\mu I^1(\mu')\ar[r]\ar[d]&\cdots\ar@{}[d]^{\qquad.}\\
0\ar[r]&\nabla(\mu^+)\ar[r]&
I^0(\mu^+)\ar[r]&
I^1(\mu^+)\ar[r]&\cdots
}
$$
We multiply all arrows in one of the rows by $-1$ to make the squares anti-commutative and then we extend the diagram to a double complex by adding zeros in all remaining rows.
Taking the total complex of this double complex gives a bounded exact complex
\begin{align*}
0\to\widetilde T_{\mu'}^\mu \nabla(\mu')\to \nabla(\mu^+)\oplus\widetilde T_{\mu'}^\mu I^0(\mu')\to\cdots\to I^i(\mu^+)\oplus\widetilde T_{\mu'}^\mu I^{i+1}(\mu')\to\cdots\,,
\end{align*}
see e.g. \cite[Ex~1.2.5]{Weib}. Using \eqref{eq.extension} we get a surjective chain map from the above complex to
$$0\to\nabla(\mu^+)\to\nabla(\mu^+)\to0\to\cdots\to0\to\cdots\,.$$
Taking the kernel we obtain an exact complex (see e.g. \cite[Ex~1.3.1]{Weib})
\begin{align}\label{eq.resolution}
0\to\nabla(\mu)\to\widetilde T_{\mu'}^\mu I^0(\mu')\to\cdots\to I^i(\mu^+)\oplus\widetilde T_{\mu'}^\mu I^{i+1}(\mu')\to\cdots\,.
\end{align}
By Proposition~\ref{prop.trans_projection} we have $\widetilde T_{\mu'}^\mu I^0(\mu')=\widetilde T_{\mu'}^\mu I_{\Lambda'}(\mu')=I_\Lambda(\mu)=I^0(\mu)$.
For $i\ge0$ we have by Propositions~\ref{prop.trans_projection} and \ref{prop.p_and_r_eqn}(ii) that
\begin{align*}
I^i(\mu^+)\oplus\widetilde T_{\mu'}^\mu I^{i+1}(\mu')&=\bigoplus_{\lambda\in\Lambda}r^i_{\lambda\mu^+}I_\Lambda(\lambda)\oplus\bigoplus_{\lambda'\in\Lambda'}r^{i+1}_{\lambda'\mu'}\widetilde T_{\mu'}^\mu I_{\Lambda'}(\lambda')\\
&=\bigoplus_{\lambda\in\Lambda} r^i_{\lambda\mu^+}I_\Lambda(\lambda)\oplus\bigoplus_{\lambda'\in\Lambda'}r^{i+1}_{\lambda'\mu'}I_\Lambda(\lambda^-)\\
&=\bigoplus_{\lambda^-\in\Lambda}(r^i_{\lambda^-\mu^+}+r^{i+1}_{\lambda'\mu'})I_\Lambda(\lambda^-)\oplus\bigoplus_{\lambda\in\Lambda,\lambda\ne\lambda^-} r^i_{\lambda\mu^+}I_\Lambda(\lambda)\\
&=\bigoplus_{\lambda^-\in\Lambda}r^{i+1}_{\lambda^-\mu^-}I_\Lambda(\lambda^-)\oplus\bigoplus_{\lambda\in\Lambda,\lambda\ne\lambda^-} r^{i+1}_{\lambda\mu^-}I_\Lambda(\lambda)\\
&=\bigoplus_{\lambda\in\Lambda} r^{i+1}_{\lambda\mu^-}I_\Lambda(\lambda)=I^{i+1}(\mu^-)=I^{i+1}(\mu)\,.
\end{align*}
If we substitute this in \eqref{eq.resolution}, then we obtain the required injective resolution of $\nabla(\mu)$ in $\mc C_\Lambda$.
\end{proof}

As in \cite[Cor~5.5]{BS} and \cite[Cor~9.3]{CdV} we obtain
\begin{corgl}
We have $r^i_{\lambda\mu}=\dim\Ext_G^i(L(\lambda),\nabla(\mu))$ for all $i\ge 0$.
Furthermore, $\Ext_G^i(L(\lambda),\nabla(\mu))=0$ unless $i\equiv n(\lambda,\mu)\ ({\rm mod}\,2)$.
\end{corgl}
\begin{proof}
By Remark~\ref{rems.polynomials}.2 all nonzero terms in $r_{\lambda\mu}$ have degree of the same parity as $n(\lambda,\mu)$.
So after applying $\Hom_G(L(\lambda),-)$ to the injective resolution in Theorem~\ref{thm.injective_resolution} the nonzero modules in the resulting complex
all have degree of the same parity as $n(\lambda,\mu)$. Therefore, all differentials in the complex are 0, i.e. the complex equals its own cohomology.
\end{proof}

\begin{corgl}
Put $t^i_{\lambda\mu}=\dim\Ext_G^i(L(\lambda),L(\mu))$ for all $i\ge0$, and \break $t_{\lambda\mu}(q)=\sum_{i\ge0}t^i_{\lambda\mu}q^i\in\mb Z[q]$.
Then $t_{\lambda\mu}=\sum_{\nu\preceq\lambda,\mu} r_{\lambda\nu}r_{\mu\nu}$.
In particular, \break $\Ext_G^i(L(\lambda),L(\mu))=0$ unless $i\equiv n(\lambda,\mu)\ ({\rm mod}\,2)$.
\end{corgl}

\begin{proof}
By the previous corollary the category $\mc C_\Lambda$ with length function $\lambda\mapsto n(\lambda)$ (see Sect~\ref{ss.arrow_diagrams}) has a Kazhdan-Lusztig theory in the sense of \cite[3.3]{CPS}. So the result follows from \cite[Cor~3.6(a)]{CPS}. See also 2.12(2), 2.13(2), 4.13(3) and C.10 in \cite[Part II]{Jan}.
\end{proof}

\begin{corgl}
We have ${\rm ch}\,L(\lambda)=\sum_{\mu\preceq\lambda}r_{\lambda\mu}(-1)\chi(\mu)$.
\end{corgl}

\begin{proof}
This follows from\cite[II.6.21(6)]{Jan} and the first corollary.
\end{proof}

\begin{remgl}
We compare $\Lambda_p$, which is the union of the various $\Lambda(s)$, with the Jantzen region of Lusztig's conjecture \cite[II.8.22]{Jan}. For the Jantzen region to be nonempty we clearly need $p\ge h-1$, where $h$ is the Coxeter number. Furthermore, the Jantzen region contains all restricted dominant weights when \hbox{$p\ge 2h-3$.} Our set $\Lambda_p$ is always nonempty, but for $G=\Sp_n$ it contains only a small portion of the restricted dominant weights and for $G=\GL_n$ only a small portion of the restricted dominant  $\SL_n$-weights lift to a weight in $\Lambda_p$.
\end{remgl}

\section{Limiting results for a fixed residue of $n$ mod $p$ }\label{s.limiting_case}
Throughout this section $\delta$ is any integer. We want to derive a certain stability result for arrow diagrams when $p\gg0$. For $\GL_n$ we cyclically shift the diagram such that the $0$-node is in the middle and then the idea is that we don't want arrows to move from one end of the diagram to the other end. Then the ``relevant $(\land,\vee)$-sequence" will be between the two walls.
For $\Sp_n$ and $\delta$ even the position of the wall relative to the first node is fixed and we don't want arrows to move around the right end node, so the ``relevant $(\land,\vee)$-sequence" is to the left of the wall. For $\Sp_n$ and $\delta$ odd the position of the wall relative to the last node is fixed and we don't want arrows to move around the left end node, so the ``relevant $(\land,\vee)$-sequence" is the second $(\land,\vee)$-sequence which comes from the arrows to the right of the wall.

\medskip

\noindent {\bf The $\GL_n$-case}. Let $r_1,r_2,s_1,s_2$ be integers $\ge0$, and let $\Lambda$ consist of pairs of partitions $(\lambda^1,\lambda^2)$
with $\lambda^i_1\le r_i$ and $l(\lambda^i)\le s_i$ for all $i\in\{1,2\}$. Choose a prime $p>2$ such that
\begin{align*}
\delta-s_1+1,1-r_2&\ge-(p-1)/2\text{\ and}\\
s_2,\delta+r_1&\le(p-1)/2\,.
\end{align*}
Choose $t\ge0$ such that $n:=\delta+tp\ge s_1+s_2$. We now identify $\Lambda$ with the set $\{[\lambda^1,\lambda^2]\,|\,(\lambda^1,\lambda^2)\in\Lambda\}$.
Then $\rho$ is defined as in Section~\ref{ss.reductive_groups}.
Now we change the labels in the arrow diagram by replacing each label by the integer in $\{-(p-1)/2,\cdots,(p-1)/2\}$ that is equal to it ${\rm mod}\ p$, and
we cyclically shift the diagram such that the first node has label $-(p-1)/2$.
Then the labels of the arrows corresponding to $[\lambda^1,\lambda^2]$ stay the same when we increase $p$, keeping $t$ (but not $n$!) fixed. They are 
$\delta+\lambda^1_1,\cdots,\delta-s_1+1+\lambda^1_{s_1}$ ($\wedge$, below the line), and 
$1-\lambda^2_1,\cdots,s_2-\lambda^2_{s_2}$ ($\vee$, above the line).
So this gives a limiting diagram with infinitely many nodes which is essentially the same as the diagram in \cite{CdV} for $(\lambda^1,\lambda^2)$ and a walled Brauer algebra $B_{u,v}(\delta)$ with $(\lambda^1,\lambda^2)$ in its label set.\footnote{Our label $(\lambda^1,\lambda^2)$ corresponds to the label $(\lambda^2,\lambda^1)$ in the notation of \cite{CdV}, and $B_{u,v}(\delta)$ corresponds to $B_{v,u}(\delta)$. In characteristic $0$ we still stick with our notation.}
See Remark~\ref{rems.polynomials}.3 how to adapt the diagram in \cite{CdV} to our conventions.
Because of the characterisation of $\preceq$ in terms of arrow diagrams in Section~\ref{ss.arrow_diagrams} it is now clear that the order $\preceq$ on $\Lambda$ is independent of $p$.

If we assume that $n\ge u+v$ and that $\Lambda$ consists of the pairs of partitions $(\lambda^1,\lambda^2)$ with $|\lambda^1|\le u$, $|\lambda^2|\le v$ and $u-|\lambda^1|=v-|\lambda^2|$, then we can use the rational Schur functor
$$f_{rat}:{\rm mod}(S(n;u,v))\to{\rm mod}(\ov B_{u,v}(\delta))\,,$$
where $S(n;u,v)$ is the rational Schur algebra and $\ov B_{u,v}(\delta)$ is the walled Brauer algebra in characteristic $p$, and deduce using arrow diagrams that for big $p$ the decomposition numbers of $\ov B_{r,s}$ are independent of $p$ and equal to the decomposition numbers of $B_{r,s}(\delta)$ in characteristic $0$. See \cite[Cor to Thm~6.1 and Prop~8.3]{T} and \cite[Thm 4.10]{CdV}.

\medskip

\noindent {\bf The $\Sp_n$-case}. Let $r,s$ be integers $\ge0$, and let $\Lambda$ consist of partitions $\lambda^1$ with $\lambda_1\le r$ and $l(\lambda)\le s$.
Choose a prime $p>2$ such that
\begin{align*}
-\delta/2-s+1&>-p/2\\ 
-\delta/2+\lambda_1&< p/2 
\end{align*}

Assume first $\delta$ is even. Then
\begin{align*}
-\delta/2-s+1&\ge-(p-1)/2\text{\ and}\\
-\delta/2+\lambda_1&\le(p-1)/2\,.
\end{align*}
Choose $t>0$ such that $m:=-\delta/2+tp\ge s$. Then $\rho$ is defined as in Section~\ref{ss.reductive_groups}.
Then we can increase $p$ and all arrows stay in the same position relative to the first node: 
Their labels are $-\delta/2+\lambda_1,-\delta/2-1+\lambda_2,\cdots,-\delta/2-s+1+\lambda_s$.
So this gives a limiting diagram with infinitely many nodes which is essentially the same as the diagram in \cite{CdV} for $\lambda$ and a Brauer algebra $B_u(\delta)$ in characteristic 0 with $\lambda$ in its label set.\footnote{The corresponding standard or irreducible module of $B_u(\delta)$ has the transpose $\lambda^T$ of $\lambda$ as label.} See Remark~\ref{rems.polynomials}.3 how to adapt the diagram in \cite{CdV} to our conventions.

Now assume $\delta$ is odd. Put $\ov\delta=\delta-p$. Then
\begin{align*}
-\ov\delta/2-s+1>0\text{\ and}\\
-\ov\delta/2+r< p\,.
\end{align*}
Choose $t\ge0$ such that $m:=-\ov\delta/2+tp\ge s$. Then $\rho$ is defined as in Section~\ref{ss.reductive_groups}.
Now change the labels by adding $p$ to the labels above the line, giving the first node top label $p$ and bottom label 0.
Then the labels of the arrows corresponding to $\lambda$ are $-\ov\delta/2+\lambda_1,-\ov\delta/2-1+\lambda_2,\cdots,-\ov\delta/2-s+1+\lambda_s$.
Then we can increase $p$ and all arrows stay in the same position relative to the last node.
We now subtract $p/2$ from all labels. Then the labels of the arrows will stay the same when we increase $p$:
$-\delta/2+\lambda_1,-\delta/2-1+\lambda_2,\cdots,-\delta/2-s+1+\lambda_s$.
If we now rotate the diagram 180 degrees we obtain the limiting diagram with infinitely many nodes which is essentially the same as the diagram in \cite{CdV} for $\lambda$ and a Brauer algebra $B_u(\delta)$ in characteristic 0 with $\lambda$ in its label set. See Remark~\ref{rems.polynomials}.3 how to adapt the diagram in \cite{CdV} to our conventions.

Now assume again that $\delta$ is arbitrary and $G=\Sp_n$. Because of the characterisation of $\preceq$ in terms of arrow diagrams in Section~\ref{ss.arrow_diagrams} it is now clear that the order $\preceq$ on $\Lambda$ is independent of $p$. If we assume that $m\ge u$ and that $\Lambda$ consists of the partitions $\lambda$ with $|\lambda|\le u$, and $u-|\lambda|$ even, then we can use the symplectic Schur functor
$$f_0:{\rm mod}(S_0(n,u))\to{\rm mod}(\ov B_u(\delta))\,,$$
where $S_0(n,u)$ is the symplectic Schur algebra and $\ov B_u(\delta)$ is the Brauer algebra in characteristic $p$, and deduce using arrow diagrams that for big $p$ the decomposition numbers of $\ov B_u$ are independent of $p$ and equal to the decomposition numbers of $B_u(\delta)$ in characteristic $0$. See \cite[Prop 2.1]{DT}, \cite[Cor to Thm~6.1]{LT} and \cite[Thm 5.8]{CdV}.

\section{Declarations}
\subsection*{Competing interests}
The author has no competing interests to declare that are relevant to the content of this article.
\subsection*{Data availability}
Data sharing is not applicable to this article as no datasets were generated or analysed during the current study.


\begin{thebibliography}{99}
\bibitem{Boe} B.~D.~Boe, {\it Kazhdan-Lusztig polynomials for Hermitian symmetric spaces}, Trans. Amer. Math. Soc. {\bf309} (1988), no. 1, 279-294.
\bibitem{BS} J.~Brundan and K.~Stroppel, {\it Highest weight categories arising from Khovanov's diagram algebra. II. Koszulity}, Transform. Groups {\bf15} (2010), no. 1, 1-45.
\bibitem{CPS} E.~Cline, B.~Parshall, L.~Scott, {\it Abstract Kazhdan-Lusztig theories}, Tohoku Math. J. (2) {\bf45} (1993), no. 4, 511-534.
\bibitem{CdV} A.~Cox and M.~de Visscher, {\it Diagrammatic Kazhdan-Lusztig theory for the (walled) Brauer algebra}, J. Algebra {\bf 340} (2011), 151-181.
\bibitem{DT} S.~Donkin and R.~Tange, {\it The Brauer algebra and the symplectic Schur algebra}, Math. Z. {\bf 265} (2010), no. 1, 187-219.
\bibitem{Jan} J.~C.~Jantzen, {\it Representations of algebraic groups}, Second edition, American Mathematical Society, Providence, RI, 2003.
\bibitem{LT} H.~Li and R.~Tange, {\it A combinatorial translation principle and diagram combinatorics for the symplectic group}, to appear in Transform. Groups.
\bibitem{Sh} A.~Shalile, {\it Decomposition numbers of Brauer algebras in non-dividing characteristic}, J. Algebra {\bf423} (2015), 963-1009.
\bibitem{T} R.~Tange, {\it A combinatorial translation principle and diagram combinatorics for the general linear group}, to appear in Transform. Groups.
\bibitem{Weib} C.~A.~Weibel, {\it An introduction to homological algebra}, Cambridge University Press, Cambridge, 1994.
\end{thebibliography}
\end{document}